\documentclass[12pt]{article}

\usepackage[margin=1.20in]{geometry}
\usepackage{sectsty}
\usepackage{xcolor}
\usepackage{amssymb,amsmath}
\usepackage{amsthm}
\usepackage{hyperref}
\usepackage{mathrsfs}
\usepackage{textcomp}
\usepackage{enumitem}
\usepackage[abbrev,nobysame,non-compressed-cites]{amsrefs}
\usepackage{todonotes}

\hypersetup{
  colorlinks,
  citecolor=black,
  filecolor=black,
  linkcolor=black,
  urlcolor=black
}

\newtheorem{Theorem}{Theorem}[section]
\newtheorem{Lemma}[Theorem]{Lemma}
\newtheorem{Proposition}[Theorem]{Proposition}

\newtheorem{Example}[Theorem]{Example}
\newtheorem{Corollary}[Theorem]{Corollary}
\newtheorem{Remark}[Theorem]{Remark}

\newcommand{\R}{\mathbb{R}}

\begin{document}

\title{\textbf{Optimal convergence rates in homogenization of Monge--Amp\`ere equations with periodic measures}}

\author{
Tianling Jin\footnote{T. Jin was partially supported by NSFC grant 12122120, and Hong Kong RGC grants GRF 16304125, GRF 16306320 and GRF 16303822.}, \quad YanYan Li\footnote{Y.Y. Li is partially supported by NSF grants DMS-2247410 and DMS-2554267.}, \quad Hung V. Tran\footnote{H.V. Tran is partially supported by NSF grant DMS-2348305.}, \quad Xushan Tu
}

\date{\today}

\maketitle

\begin{abstract}
\setlength{\emergencystretch}{1em}
We establish optimal convergence rates in periodic homogenization of Monge--Amp\`ere equations, allowing the microscopic measure to be singular or degenerate. Let $\mu$ be a nonnegative $\mathbb Z^n$-periodic Borel measure with $\mu(\mathbb T^n)=1$, and set $\mu_\varepsilon(E)=\varepsilon^n\mu(E/\varepsilon)$. On a bounded convex domain $\Omega\subset\R^n$, let $u^\varepsilon$ and $u$ be the convex solutions of
\[
\det D^2u^\varepsilon=F(x,x/\varepsilon)\,\mu_\varepsilon,\qquad\det D^2u=\overline{F}(x)
\]
with the same Dirichlet data, where the coefficient $F$ is bounded, uniformly positive, and periodic in its second variable, and $\overline F(x)=\int_{\mathbb T^n}F(x,y)\,d\mu(y)$. For $m=0,1$ and $0<\alpha\leq1$, we prove
\[
\|u^\varepsilon-u\|_{L^\infty(\Omega)}\leq C\varepsilon^{m+\alpha},\qquad \forall\,0<\varepsilon\leq1,
\]
provided that $u\in C^{m+2,\alpha}(\overline\Omega)$ and $F(\cdot,y)\in C^{m,\alpha}(\overline\Omega)$ uniformly in $y$. These exponents are optimal in their respective regularity scales, and the constant $C$ is independent of the distribution of the microscopic measure. In particular, the estimates include the optimal quadratic rate for general periodic measures.

Our proof of the lower-order estimate ($m=0$) relies on uniform Lipschitz estimates for periodic correctors. For the higher-order estimate ($m=1$), we introduce a constrained Monge--Amp\`ere problem that produces convex approximations with prescribed cell averages. This construction captures the cancellation in the third-order corrector expansion without uniform Hessian bounds for the correctors.

\medskip
\noindent{\it Keywords}: Monge--Amp\`ere equations, periodic homogenization,  convergence rates

\medskip
\noindent {\it MSC (2010)}: Primary 35B27; Secondary 35J96, 35B40.
\end{abstract}

\tableofcontents
\bigskip

\section{Introduction}
\subsection{Main results}

In this paper, we study quantitative homogenization for the Monge--Amp\`ere equation with locally weighted periodic measures on a bounded convex domain $\Omega\subset\R^n$. Let $\mathbb T^n:=\mathbb R^n/\mathbb Z^n$ and let $\varphi\in C(\partial\Omega)$.

Let $\mu$ be a nonnegative locally finite $\mathbb Z^n$-periodic Borel measure normalized by
\[
\mu(\mathbb T^n)=1.
\]
Define
\[
\mu_\varepsilon(E):=\varepsilon^n\mu(E/\varepsilon)
\]
for Borel sets $E\subset\mathbb R^n$. Let $F:\overline\Omega\times\mathbb T^n\to(0,\infty)$ be jointly Borel and periodic in its second variable,  and set
\[
\overline F(x):=\int_{\mathbb T^n}F(x,y)\,d\mu(y).
\]
Let $u^\varepsilon$ be the convex solution of the oscillatory problem:
\begin{equation}
\label{eq:weighted-oscillatory-problem}
\det D^2u^\varepsilon=F(x,x/\varepsilon)\,\mu_\varepsilon\quad\text{in }\Omega,\qquad u^\varepsilon=\varphi \quad\text{on }\partial\Omega.
\end{equation}
When the weighted periodic measures $F(x,x/\varepsilon)\,\mu_\varepsilon$ converge weakly on $\overline\Omega$ to $\overline F(x)\,dx$ as suggested by periodic averaging, the Alexandrov maximum principle and convexity imply that every subsequence of $u^\varepsilon$ with $\varepsilon\to0$ admits a subsequence converging uniformly to a convex function $u$, and consequently the corresponding Monge--Amp\`ere measures converge weakly to that of $u$. The limit \(u\) then solves the effective problem
\begin{equation}
\label{eq:weighted-effective-problem}
\det D^2u=\overline F(x) \quad\text{in }\Omega,\qquad u=\varphi \quad\text{on }\partial\Omega.
\end{equation}
By the uniqueness of solutions to \eqref{eq:weighted-effective-problem}, the whole family $u^\varepsilon$ converges uniformly to $u$ as $\varepsilon\to0$. Our goal is to quantify the  convergence rate of $u^\varepsilon$ to $u$ in terms of the regularity of this homogenized solution.


\begin{Theorem}
\label{thm:weighted-measure-holder-rate}
Let \(\Omega\subset\mathbb R^n\) be a bounded convex domain. Let $m\in\{0,1\}$, $\alpha\in(0,1]$, and $F\in C(\overline\Omega\times\mathbb T^n)$. Assume that $0<\lambda\leq F\leq\Lambda$. Let $u^\varepsilon$ and $u$ be the solutions of \eqref{eq:weighted-oscillatory-problem} and \eqref{eq:weighted-effective-problem}, respectively. If $F(\cdot,y)\in C^{m,\alpha}(\overline\Omega)$ uniformly in $y$ and $u\in C^{m+2,\alpha}(\overline\Omega)$, then
\begin{equation*}
\|u^\varepsilon-u\|_{L^\infty(\Omega)}\leq C\varepsilon^{m+\alpha},\qquad \forall\,0<\varepsilon\leq1.
\end{equation*}
The constant $C$ depends only on $n$, $\alpha$, $\lambda$, $\Lambda$, the inradius of $\Omega$, $\operatorname{diam}\Omega$, and the indicated norms of $F$ and $u$, and is independent of the distribution of $\mu$.
\end{Theorem}

The same conclusion holds for jointly Borel measurable $F$, provided that the stated two-sided bounds and uniform $C^{m,\alpha}(\overline\Omega)$ bounds hold for every $y$ outside a fixed $\mu$-null subset of $\mathbb T^n$.

\begin{Remark}
\label{rem:boundary-regularity}
For $0<\alpha<1$ and a uniformly convex domain $\Omega$, the regularity assumptions on $u$ in Theorem~\ref{thm:weighted-measure-holder-rate} can be obtained from the boundary regularity theory for \eqref{eq:weighted-effective-problem}. For $m=0$, the global Schauder estimates of Trudinger and Wang \cite{trudinger2008boundary} apply when $\partial\Omega\in C^3$ and $\varphi\in C^3(\overline\Omega)$. For $m=1$, assuming $\partial\Omega\in C^{3,\alpha}$ and $\varphi\in C^{3,\alpha}(\overline\Omega)$, Schauder bootstrapping gives $u\in C^{3,\alpha}(\overline\Omega)$; see \cite{caffarelli1984dirichlet}, \cite{gilbarg2001elliptic}*{Chapter~17}, and \cite{Le2024}*{Chapter~4}. See also Savin \cite{savin2013pointwise} for the pointwise boundary theory. Under these sufficient assumptions, the constant $C$ in Theorem~\ref{thm:weighted-measure-holder-rate} depends only on $n, \alpha, \lambda, \Lambda$, the lower and upper bounds of the curvatures of $\partial\Omega$, the indicated norm of $F$, and the indicated norms of $\partial\Omega$ and $\varphi$ together with the moduli of continuity of their third derivatives.
\end{Remark}

\begin{Remark}
\label{rem:main-sharpness}
The quadratic endpoint is sharp in one dimension: for $\mu=dy$, $\Omega=(0,1)$, $\varphi=0$, $F(x,y)=1+\frac12\cos(2\pi y)$, and $\varepsilon=N^{-1}$ with $N\in\mathbb N$, direct integration yields
\[
\|u^\varepsilon-u\|_{L^\infty(0,1)}=\frac{1}{4\pi^2}\varepsilon^2.
\]
The exponents in Theorem~\ref{thm:weighted-measure-holder-rate} are optimal in the corresponding regularity scales, already in one dimension; see Example~\ref{ex:fixed-data-sharpness}.
\end{Remark}

When $\mu=dy$, Theorem~\ref{thm:weighted-measure-holder-rate} gives both estimates in Theorem~\ref{thm:main-result} without assuming H\"older continuity in $y$. Example~\ref{ex:fixed-data-sharpness} shows that the exponents $m+\alpha$ are sharp. Taking $F(x,y)=F(x)$ gives:

\begin{Corollary}
\label{cor:fixed-measure-holder-rate}
Let $m\in\{0,1\}$ and $\alpha\in(0,1]$, and let $\mu$ be a normalized nonnegative locally finite $\mathbb Z^n$-periodic Borel measure. Let $F\in C^{m,\alpha}(\overline\Omega)$ satisfy $0<\lambda\leq F\leq\Lambda$, and let $u^\varepsilon$ and $u$ be the convex solutions of
\[
\det D^2u^\varepsilon=F(x)\mu_\varepsilon,\quad\det D^2u=F(x) \quad\text{in }\Omega,\qquad u^\varepsilon=u=\varphi \quad\text{on }\partial\Omega.
\]
If $u\in C^{m+2,\alpha}(\overline\Omega)$, then
\[
\|u^\varepsilon-u\|_{L^\infty(\Omega)}\leq C\varepsilon^{m+\alpha},\qquad \forall\,0<\varepsilon\leq1.
\]
The constant $C$ depends only on $n$, $\alpha$, $\lambda$, $\Lambda$, the inradius of $\Omega$, $\operatorname{diam}\Omega$, $\|F\|_{C^{m,\alpha}(\overline\Omega)}$, and $\|u\|_{C^{m+2,\alpha}(\overline\Omega)}$.
\end{Corollary}

\begin{Remark}
We very recently became aware of the preprint \cite{CGLWZ} by Cui, Gu, Le, Wang, and Zhou. In the setting $F\equiv 1$ and $\mu_\varepsilon=f(\frac{x}{\varepsilon})\,dx$, under the assumptions that $0<\lambda\leq f\leq \Lambda$, $\varphi\in C^3(\overline\Omega)$ and $\Omega$ is a bounded smooth uniformly convex domain, they independently establish the $\varepsilon$-convergence rate appearing in Corollary \ref{cor:fixed-measure-holder-rate}. In addition, they obtain an $\varepsilon$-convergence rate in a H\"older norm for $u^\varepsilon-u$.
\end{Remark}

\begin{Remark}
\label{rem:general-periodicity}
All the results remain valid if the $\varepsilon\mathbb Z^n$-periodic measure $\mu_\varepsilon$ is replaced by a normalized $\Gamma_\varepsilon$-periodic measure, where $\Gamma_\varepsilon$ admits a fundamental cell $Q_\varepsilon$ satisfying $\operatorname{diam}Q_\varepsilon\leq\varepsilon$, with the periodicity of $F(x,\cdot)$ adjusted accordingly. This follows by passing to a larger period cell comparable to a cube and applying a linear change of variables.
\end{Remark}

\subsection{Related work}
The literature on homogenization is vast. Here we briefly recall some results on periodic homogenization of linear equations in nondivergence form and fully nonlinear uniformly elliptic equations.  For periodic linear elliptic equations in nondivergence form, Avellaneda and Lin \cite{avellaneda1989compactness} established regularity estimates uniform in the microscopic period. Evans \cite{evans1992periodic} proved periodic homogenization for fully nonlinear elliptic equations using the perturbed test function method. Caffarelli and Souganidis \cite{caffarelli2010souganidis} obtained quantitative homogenization rates under mixing assumptions, with algebraic rates in the periodic setting governed by a generally nonexplicit exponent. They also obtained quantitative rates for monotone finite-difference approximations of fully nonlinear uniformly elliptic equations in \cite{caffarelli2008souganidis}. Under higher-order smoothness assumptions, Kim and Lee \cite{kim2016higher} constructed higher-order expansions for periodic linear and fully nonlinear uniformly elliptic equations using interior and boundary layer correctors. For periodic linear uniformly elliptic equations in nondivergence form, Guo, Tran, and Yu \cite{GTY2020} proved that $O(\varepsilon)$ is the optimal rate. When the diffusion matrix $A$ has only one degree of freedom, that is $A(y)=C+a(y)M$ where $C, M$ are constant symmetric matrices and $a$ is a scalar periodic function, the rate is improved to $O(\varepsilon^2)$ by Guo, Sprekeler, and Tran \cite{GST2025}.

This uniformly elliptic framework does not apply directly to the Monge--Amp\`ere equation, which is not uniformly elliptic in the class of convex solutions. Quantitative estimates must instead exploit the geometry of convex solutions and their sections. For positive continuous periodic densities, Caffarelli and Li \cite{caffarelli2004liouville} proved an algebraic homogenization estimate. Their method can also be used to extend this estimate to normalized nonnegative locally finite periodic Borel measures; see Theorem~2.3 in \cite{jin2025liouville}. In a related discrete setting, Chen, Huang, and Wang \cite{chen2019convergence} established convergence rates for polyhedral approximations of Alexandrov solutions, including the rates $\delta^\alpha$, $\delta^{1+\alpha}$, and $\delta^2$ in terms of the mesh size $\delta$, under successively stronger regularity assumptions. Although the approximation mechanism is different, the same hierarchy of rates appears in our homogenization estimates.

\subsection{Ideas of the proof}
To explain the mechanism behind the two convergence rates, we first take $\mu=dy$, assume that the data are smooth, and consider the formal expansion
\begin{equation}
\label{eq:formal-third-order-expansion}
u^\varepsilon(x)=u(x)+\varepsilon a(x)+\varepsilon^2v(x,y)+\varepsilon^3\psi(x,y)+\cdots,
\end{equation}
where $y=x/\varepsilon$, $v,\psi$ are periodic in $y$, and we take $a$ to be independent of $y$ to avoid an $\varepsilon^{-1}$ Hessian term. Substituting this expansion into \eqref{eq:weighted-oscillatory-problem} and matching the zeroth-order terms, we obtain
\[
\det\bigl(D^2u(x)+D_y^2v(x,y)\bigr)=F(x,y).
\]
Since $\det D^2u=\overline F$, we choose $v(x,\cdot)$ to be the zero-mean periodic solution of this cell problem. Collecting the order-$\varepsilon$ terms in $\det D^2u^\varepsilon$ then yields
\[
W^{ij}\bigl(a_{ij}+v_{x_i y_j}+v_{x_j y_i}+\psi_{y_i y_j}\bigr)=0,
\]
where $W=\operatorname{cof}(D^2u+D_y^2v)$. Integrating over $\mathbb T^n$ and using integration by parts, the terms containing $y$-derivatives vanish by periodicity and the divergence-free property of $W$. Using $\int_{\mathbb T^n}W\,dy=\operatorname{cof}D^2u$, we now obtain
\[
(\operatorname{cof}D^2u)^{ij}a_{ij}=0\quad\text{in }\Omega.
\]
Matching the boundary values at order $\varepsilon$ gives $a=0$ on $\partial\Omega$, so the maximum principle yields $a\equiv0$. We therefore determine $\psi(x,\cdot)$ from
\begin{equation}
\label{eq:formal-third-corrector}
W^{ij}\psi_{y_i y_j}=-2W^{ij}v_{x_i y_j}\quad\text{in }\mathbb T^n,\qquad\int_{\mathbb T^n}\psi(x,y)\,dy=0,
\end{equation}
where the last equation fixes the additive constant. For linear nondivergence-form uniformly elliptic equations, the coefficient $a$ need not vanish, and the $O(\varepsilon)$ convergence rate is generally optimal; see \cites{GST2025,GTY2020}. Here the divergence-free cofactor identity gives the additional cancellation that allows the higher-order rate. For the rigorous approximation, we mollify the background Hessian and the normalized density in the $x$ variable at scale $\varepsilon$ and construct the corresponding correctors $v^\varepsilon$ and $\psi^\varepsilon$. Combining the mollification error with the derivative estimates for these correctors, we control the determinant error and prove the two convergence rates.

To convert these approximations into convergence estimates, we use the two comparison principles established in Section~\ref{sec:periodic-cell-problems}. For positive H\"older densities, Proposition~\ref{prop:approximate-solution-principle} controls the homogenization error through the errors in the values and determinant of a smooth uniformly convex global approximation. For general periodic measures, the available corrector estimates do not provide the regularity and convexity needed for this argument. We therefore use the local comparison principle in Proposition~\ref{prop:local-comparison}, which compares a continuous approximation of $u$ with convex potentials constructed locally. Under local quadratic growth bounds, this principle controls the global homogenization error through local value errors and relative Monge--Amp\`ere measure errors.  The continuous approximation need not be convex, and no Hessian bounds for the local potentials or uniform positive lower bound on $r$ are required. This formulation is used in both convergence estimates for general periodic measures.

For the lower-order estimate $m=0$ in Theorem~\ref{thm:weighted-measure-holder-rate}, we freeze the background Hessian and periodic weight at each point and compare locally with the resulting convex potentials. Controlling the error introduced by freezing these parameters requires uniform Lipschitz dependence of the cell correctors on them. We establish this estimate in Theorem~\ref{thm:measure-corrector-lipschitz}, using the almost H\"older estimate for the linearized equation from our previous work~\cite{jin2025liouville}. This stability estimate allows us to treat degenerate or singular periodic measures without uniform second-derivative estimates for the correctors.

The proof of the higher-order estimate $m=1$ in Theorem~\ref{thm:weighted-measure-holder-rate} requires substantial additional work. Guided by the local comparison principle Proposition~\ref{prop:local-comparison}, we work at the unit microscopic scale (retaining the notation $u$ after localization and rescaling),  and seek a convex potential $\widetilde u$ satisfying
\begin{equation}
\label{eq:intro-target-estimates}
|\det D^2\widetilde u-G_\varepsilon f|\leq C\varepsilon^{1+\alpha}f,\qquad \|\widetilde u-u\|_\infty\leq C,
\end{equation}
where $f$ is a normalized periodic density, and $G_\varepsilon$ is a positive weight close to a constant. A natural starting point is the third-order expansion \eqref{eq:formal-third-order-expansion}. The lack of uniform regularity bounds on the microscopic data, however, prevents us from directly controlling its convexity and determinant error.

We therefore introduce a local averaging operator $\Pi$, defined by $\Pi[w](z)=\int_{\mathbb R^n}\chi_z w f\,dx$ for fixed smooth cell weights $\chi_z$. Prescribing $\Pi[\widetilde u-u]=0$ separates the cell averages from the variations in $\ker\Pi$ that remain available to adjust the density. The average constraint, together with convexity and the bounded Hessian of $u$, gives the required bound for $\widetilde u-u$, as shown in Lemma~\ref{lem:cell-average-convex-bound}. Moreover, the corrector terms in the frozen form of \eqref{eq:formal-third-order-expansion} can be normalized to have zero cell averages, as in \eqref{eq:frozen-expansion-average-compatibility}, so the constraints are compatible with its first-order cancellation.

The prescribed cell averages may preclude an exact solution of $\det D^2\widetilde u=G_\varepsilon f$. We therefore  retain the average constraint and impose this equation only against test functions in $\ker\Pi$:
\[
\langle\det D^2\widetilde u-G_\varepsilon f,\zeta\rangle_{L^2}=0\qquad\forall\,\zeta\in\ker\Pi.
\]
Since $\Pi$ is defined by the weights $\chi_z f$, this condition forces the density error to be a linear combination of these weights. Writing
\[
G_\varepsilon=\sum_z\det D^2u(\bar x_z)\chi_z+\omega,
\]
where $\bar x_z$ are the weighted cell barycenters, we obtain the constrained problem
\begin{equation}
\label{eq:intro-constrained-problem}
\det D^2\widetilde u=\left(\sum_z\lambda(z)\chi_z+\omega\right)f,\qquad \Pi[\widetilde u-u]=0,
\end{equation}
where $\lambda(z)$ are the Lagrange multipliers for the cell-average constraints. The remaining density error is therefore determined by the differences $\lambda(z)-\det D^2u(\bar x_z)$. The cell multipliers $\lambda$ admit uniform first- and second-variation estimates along the data paths, without requiring corresponding higher-order  regularity for $\widetilde u$.

We solve \eqref{eq:intro-constrained-problem} on large tori for small perturbations of a quadratic background. Linearizing with respect to the background $u$ and the additive density $\omega$, and denoting first variations by dots, we obtain the coupled linearized system
\[
\mathcal L[\dot{\widetilde u}]=\left(\sum_z\dot\lambda(z)\chi_z+\dot\omega\right)f,\qquad \Pi[\dot{\widetilde u}-\dot u]=0,
\]
where $\mathcal L[w]=\widetilde U^{ij}w_{ij}$ with $(\widetilde U^{ij}):=\operatorname{cof}(D^2\widetilde u)$. By establishing uniform energy and pointwise estimates for $\dot{\widetilde u}-\dot u$ and $\dot\lambda$, with constants independent of the upper and lower bounds and derivatives of $f$, we apply the implicit function theorem and the continuity method to construct the constrained potential and establish its stability.

We freeze the data at each cell barycenter and use the third-order corrector defined by \eqref{eq:formal-third-corrector} to construct a model for the first variation. Its multiplier vanishes at the frozen cell. Comparing the first variation with this model and estimating the second variation, we prove
\[
\sup_z\left|\lambda(z)-\det D^2u(\bar x_z)\right|\leq C\varepsilon^{1+\alpha}.
\]
This bounds the density error. The average constraint and convexity control the function-value error, giving \eqref{eq:intro-target-estimates}.

Finally, we apply the local comparison principle in Proposition \ref{prop:local-comparison} to complete the proof.

\section{Preliminaries}
\label{sec:periodic-cell-problems}

Throughout the paper, $c$ and $C$ denote positive constants whose values may change from line to line. We use the Einstein summation convention. Equations of the form $\det D^2u=f$ or $\det D^2u=\mu$ are understood in the Alexandrov sense. When no confusion arises, we identify a density $f$ with the measure $f\,dx$.

\subsection{Periodic solutions}
\label{subsec:periodic-solutions}

Let $\mathcal S^{n\times n}$ denote the space of symmetric $n\times n$ matrices and let $\mathcal S_+^{n\times n}$ denote its subset of positive definite matrices. Let $\nu$ be a nonzero nonnegative locally finite $\mathbb Z^n$-periodic Borel measure. A \emph{periodic solution} of $\det D^2w=\nu$ in $\mathbb R^n$ is a convex solution of the form
\[
w(y)=q_A(y)+\ell(y)+v(y),\qquad q_A(y):=\frac12y\cdot(Ay),
\]
where $A\in\mathcal S_+^{n\times n}$, $\ell$ is affine, and $v$ is $\mathbb Z^n$-periodic. The matrix $A$ necessarily satisfies the compatibility condition
\[
\det A=\nu(\mathbb T^n).
\]
Indeed, for every $\zeta\in C^2(\mathbb T^n)$,
\begin{equation}
\label{eq:periodic-cell-compatibility}
\int_{\mathbb T^n}\det(A+D^2\zeta)\,dy=\det A.
\end{equation}
Applying \eqref{eq:periodic-cell-compatibility} to periodic mollifications of $v$ and passing to the limit gives the compatibility condition.

Conversely, for each $A\in\mathcal S_+^{n\times n}$ satisfying this condition, there exists a periodic function $v$ such that $q_A+v$ is convex and solves $\det D^2(q_A+v)=\nu$ in $\mathbb R^n$; $v$ is unique up to an additive constant \cite{jin2025liouville}*{Theorem~1.2}. Moreover, every entire convex solution of $\det D^2w=\nu$ has the above form \cite{jin2025liouville}*{Theorem~1.3}.

From now on, we assume $\nu(\mathbb T^n)=1$. For each $A\in\mathcal S_+^{n\times n}$, let $w_{A,\nu}=q_A+v_{A,\nu}$ be the unique normalized periodic solution of
\begin{equation*}
\det D^2w_{A,\nu}=(\det A)\nu\quad\text{in }\mathbb R^n,
\end{equation*}
where $v_{A,\nu}$ is $\mathbb Z^n$-periodic and $\int_{\mathbb T^n}v_{A,\nu}(y)\,dy=0$. By \cite{li1990existence}*{Lemma~2.2 and the argument on pp.~249--250},
\begin{equation}
\label{eq:measure-corrector-size}
\|v_{A,\nu}\|_{L^\infty(\mathbb T^n)}+\operatorname{Lip}(v_{A,\nu})\leq C_n\operatorname{tr}A.
\end{equation}
Here $C_n$ depends only on $n$. These bounds, together with weak continuity of Monge--Amp\`ere measures and uniqueness, imply that, for each fixed $A\in\mathcal S_+^{n\times n}$,
\[
v_{A,\nu_k}\to v_{A,\nu}\quad\text{uniformly on }\mathbb T^n\quad\text{whenever }\nu_k\rightharpoonup\nu.
\]

When $\nu=\rho\,dy$, we write $v_{A,\rho}$ and $w_{A,\rho}$ for $v_{A,\rho\,dy}$ and $w_{A,\rho\,dy}$. For a fixed reference measure $\mu$, we omit it from the subscripts and write $v_A$ and $w_A$.

\subsection{A comparison principle}
\label{sec:comparison-principle}

We shall use the following consequence of the comparison principle.

\begin{Proposition}
\label{prop:approximate-solution-principle}
Fix $\beta\in(0,2]$ and $c_0,C_0>0$. There exists $\varepsilon_0>0$, depending only on $n,\beta,c_0,C_0$, such that the following holds for every $0<\varepsilon<\varepsilon_0$. Let $u^\varepsilon$ and $u$ be the convex solutions of \eqref{eq:weighted-oscillatory-problem} and \eqref{eq:weighted-effective-problem}, with $\mu=dy$ and $0<\lambda\leq F\leq\Lambda$. Suppose there exists $V^\varepsilon\in C^2(\Omega)\cap C(\overline\Omega)$ such that
\[
\|V^\varepsilon-u\|_{L^\infty(\overline\Omega)}\leq C_0\varepsilon^\beta,\qquad c_0I\leq D^2V^\varepsilon\leq C_0I,\qquad\left|\det D^2V^\varepsilon-F(x,x/\varepsilon)\right|\leq C_0\varepsilon^\beta
\]
in $\Omega$. Then,
\[
\|u^\varepsilon-u\|_{L^\infty(\Omega)}\leq C\varepsilon^\beta.
\]
The constant $C$ depends only on $n$, $c_0$, $C_0$, and $\operatorname{diam}\Omega$.
\end{Proposition}

\begin{proof}
After a translation, we may assume that $0\in\Omega$. Set
\[
W^\pm=V^\varepsilon\pm c\varepsilon^\beta|x|^2.
\]
Since $D^2W^\pm=D^2V^\varepsilon\pm2c\varepsilon^\beta I$, we may choose $c$ sufficiently large so that, for sufficiently small $\varepsilon$, both functions are convex and
\[
\det D^2W^+\geq\det D^2u^\varepsilon\geq\det D^2W^-.
\]
On $\partial\Omega$, where $u^\varepsilon=u$, the bound $\|W^\pm-u\|_{L^\infty(\overline\Omega)}\leq C\varepsilon^\beta$ gives
\[
W^+-C\varepsilon^\beta\leq u^\varepsilon\leq W^-+C\varepsilon^\beta\quad\text{on }\partial\Omega.
\]
The comparison principle gives the same inequalities in $\Omega$. Combining these inequalities with $\|W^\pm-u\|_{L^\infty(\overline\Omega)}\leq C\varepsilon^\beta$, we obtain $\|u^\varepsilon-u\|_{L^\infty(\Omega)}\leq2C\varepsilon^\beta$, as claimed.
\end{proof}

We also use a local version of this comparison argument.

\begin{Proposition}
\label{prop:local-comparison}
Fix $\beta>0$ and $C_0>0$. There exists $\varepsilon_0>0$, depending only on $n,\beta,C_0$, such that the following holds for every $0<\varepsilon<\varepsilon_0$. Let $u^\varepsilon$ and $u$ be the convex solutions of \eqref{eq:weighted-oscillatory-problem} and \eqref{eq:weighted-effective-problem}, respectively. Let $V^\varepsilon\in C(\overline\Omega)$ satisfy
\[
\|V^\varepsilon-u\|_{L^\infty(\overline\Omega)}\leq C_0\varepsilon^\beta.
\]
Suppose that for every $x_*\in\Omega$ there exist a radius $r=r(x_*)\in(0,1]$, an affine function $\ell_{x_*}$, and a convex function $W_{x_*}\in C(\overline D_{x_*})$, where $D_{x_*}:=\Omega\cap B_r(x_*)$, such that
\begin{equation}
\label{eq:local-comparison-hypotheses}
\begin{gathered}
|W_{x_*}-\ell_{x_*}|\leq C_0r^2,\qquad|W_{x_*}-V^\varepsilon|\leq C_0\varepsilon^\beta r^2,\\
(1-C_0\varepsilon^\beta)\det D^2u^\varepsilon\leq\det D^2W_{x_*}\leq(1+C_0\varepsilon^\beta)\det D^2u^\varepsilon
\end{gathered}
\end{equation}
in $D_{x_*}$. Then
\[
\|u^\varepsilon-u\|_{L^\infty(\Omega)}\leq C\varepsilon^\beta.
\]
The constant $C$ depends only on $n$, $C_0$, and $\operatorname{diam}\Omega$.
\end{Proposition}

\begin{proof}
By translation and scaling, we may assume that $\Omega\subset B_1$. Set $M:=\max_{\overline\Omega}|u^\varepsilon-V^\varepsilon|$. If $M\leq4C_0\varepsilon^\beta$, the conclusion follows by the triangle inequality. We therefore assume that $M>4C_0\varepsilon^\beta$. For each $x_*\in\Omega$, define
\[
\widetilde W_{x_*}^\pm(x)=\ell_{x_*}(x)+(1\mp C_0\varepsilon^\beta)^{-1/n}\bigl(W_{x_*}(x)-\ell_{x_*}(x)\bigr)\pm\frac12M(2x_*\cdot x-|x_*|^2).
\]
By \eqref{eq:local-comparison-hypotheses}, we have in $D_{x_*}$ that
\[
\det D^2\widetilde W_{x_*}^-\leq\det D^2u^\varepsilon\leq\det D^2\widetilde W_{x_*}^+
\]
and
\[
\left|\widetilde W_{x_*}^\pm(x)-V^\varepsilon(x)\mp\frac12M(2x_*\cdot x-|x_*|^2)\right|\leq C\varepsilon^\beta r(x_*)^2.
\]

We first consider the case $M=\max_{\overline\Omega}(u^\varepsilon-V^\varepsilon)$. Since $|u^\varepsilon-V^\varepsilon|\leq C_0\varepsilon^\beta<M/4$ on $\partial\Omega$, a maximum point $x_0$ of $u^\varepsilon-V^\varepsilon+\frac12M|x|^2$ lies in $\Omega$. Write $r_0=r(x_0)$. By the maximum property and the preceding value estimate, we have
\[
(u^\varepsilon-\widetilde W_{x_0}^-)(x_0)-(u^\varepsilon-\widetilde W_{x_0}^-)(x)\geq\frac12M|x-x_0|^2-C\varepsilon^\beta r_0^2\qquad\forall\,x\in\overline D_{x_0}.
\]
On $\overline\Omega\cap\partial B_{r_0}(x_0)$, the right-hand side is $r_0^2(\frac12M-C\varepsilon^\beta)$. On $\partial\Omega\cap\overline B_{r_0}(x_0)$, the boundary estimate and the fact that the maximum of $u^\varepsilon-V^\varepsilon+\frac12M|x|^2$ is at least $M$ give
\[
(u^\varepsilon-\widetilde W_{x_0}^-)(x_0)\geq M-C\varepsilon^\beta r_0^2,\qquad(u^\varepsilon-\widetilde W_{x_0}^-)(x)\leq\frac12M+C_0\varepsilon^\beta+C\varepsilon^\beta r_0^2.
\]
Since $r_0\leq1$ and $C_0\varepsilon^\beta<M/4$, combining these estimates on the two parts of $\partial D_{x_0}$ with the comparison principle gives
\[
r_0^2\left(\frac14M-C\varepsilon^\beta\right)\leq (u^\varepsilon-\widetilde W_{x_0}^-)(x_0)-\sup_{\partial D_{x_0}}(u^\varepsilon-\widetilde W_{x_0}^-)\leq0.
\]
Hence $M\leq C\varepsilon^\beta$.

The remaining case $M=\max_{\overline\Omega}(V^\varepsilon-u^\varepsilon)$ follows by the same argument, taking a maximum point $x_0$ of $V^\varepsilon-u^\varepsilon+\frac12M|x|^2$ and replacing $\widetilde W_{x_0}^-$ by $\widetilde W_{x_0}^+$. The conclusion follows from $\|u^\varepsilon-u\|_{L^\infty(\Omega)}\leq M+C_0\varepsilon^\beta$.
\end{proof}

\subsection{Approximation of periodic measures}
\label{sec:measure-approximation}

In proving our main results, we shall work with smooth data defined on a neighborhood of $\overline\Omega$. The following extension and approximation result justifies this reduction.

\begin{Lemma}
\label{lem:smooth-measure-approximation}
\label{sec5:smooth-approximation}
Let $m\in\{0,1\}$ and $\alpha\in(0,1]$. Assume that $0<\lambda\leq F\leq\Lambda$, $F(\cdot,y)\in C^{m,\alpha}(\overline\Omega)$ uniformly in $y$, and $u\in C^{m+2,\alpha}(\overline\Omega)$. There exist a fixed neighborhood $\mathcal O$ of $\overline\Omega$, smooth uniformly convex functions $u_k$ on $\mathcal O$, smooth positive periodic densities $f_k$ of mass one, and coefficients $F_k$ smooth in all variables and periodic in $y$, such that, writing $d\mu_k=f_k\,dy$,
\[
\int_{\mathbb T^n}F_k(x,y)\,d\mu_k(y)=\det D^2u_k(x)\qquad\text{in }\mathcal O.
\]
The positive bounds for $F_k$, the convexity bounds for $u_k$, and their respective $C^{m,\alpha}$ and $C^{m+2,\alpha}$ norms on $\mathcal O$ are controlled uniformly by the original data, with the norm of $F_k(\cdot,y)$ taken uniformly in $y$. Moreover, $u_k\to u$ in $C^{m+2}(\overline\Omega)$, and for every fixed $\varepsilon>0$,
\[
F_k(x,x/\varepsilon)f_k(x/\varepsilon)\,dx\rightharpoonup F(x,x/\varepsilon)\,d\mu_\varepsilon(x)\quad\text{locally in }\Omega.
\]
\end{Lemma}

\begin{proof}
Set $g=F/\overline F$. $u$ is uniformly convex. Using bounded Whitney extension operators, we extend $u$ and $g(\cdot,y)$ to a fixed neighborhood $\mathcal V$ of $\overline\Omega$, with controlled norms. The same operator for every $y$ preserves joint Borel measurability. After shrinking $\mathcal V$, uniform convexity and the positive bounds are preserved. Dividing the extended $g$ by its $\mu$-average, we may also assume that
\[
\int_{\mathbb T^n}g(x,y)\,d\mu(y)=1\qquad\text{in }\mathcal V.
\]
We set $F=(\det D^2u)g$ on $\mathcal V$. These extensions agree with the original data on $\overline\Omega$.

Choose a fixed neighborhood $\mathcal O$ with $\overline\Omega\subset\mathcal O\Subset\mathcal V$. Let $(\kappa_k)$ be standard mollifiers in $x$, with sufficiently small supports, and let $(\eta_k)$ be smooth strictly positive periodic approximate identities. We set
\begin{equation*}
\begin{aligned}
u_k&=u*\kappa_k=\int_{\mathbb R^n}\kappa_k(s)\,u(x-s)\,ds,\\
f_k&=\mu*\eta_k=\int_{\mathbb T^n}\eta_k(y-z)\,d\mu(z),\\
g_k(x,y)&=\frac{\bigl[\kappa_k*_x((g\mu)*_y\eta_k)\bigr](x,y)}{f_k(y)}=\frac{\displaystyle\int_{\R^n}\kappa_k(s)\int_{\mathbb T^n}\eta_k(y-z)g(x-s,z)\,d\mu(z)\,ds}{\displaystyle\int_{\mathbb T^n}\eta_k(y-z)\,d\mu(z)},\\
F_k(x,y)&=\det D^2u_k(x)\,g_k(x,y).
\end{aligned}
\end{equation*}
The convolutions preserve the normalization of $g$ and the mass of $\mu$, so
\[
\int_{\mathbb T^n}g_k(x,y)\,d\mu_k(y)=1.
\]
This proves the compatibility relation. The stated bounds follow from the extension bounds and positivity of the convolution kernels, while $u_k\to u$ in $C^{m+2}(\overline\Omega)$ follows from mollification.

Fix $\varepsilon>0$ and $\zeta\in C_c^\infty(\Omega)$, and write $a_k=\det D^2u_k$ and $Q=[-\frac12,\frac12)^n$. By periodicity and Fubini's theorem,
\[
\begin{aligned}
&\int_{\R^n}\zeta(x)F_k(x,x/\varepsilon)f_k(x/\varepsilon)\,dx\\
&=\int_{\mathbb R^n}\int_Q\kappa_k(s)\eta_k(h)\int_{\R^n}\zeta(x+\varepsilon h)a_k(x+\varepsilon h)g(x+\varepsilon h-s,x/\varepsilon)\,d\mu_\varepsilon(x)\,dh\,ds.
\end{aligned}
\]
Here the test products are extended by zero outside their domains. Since $a_k\to\det D^2u$ uniformly and $g$ is uniformly continuous in $x$, uniformly in $y$, the above integral converges to $\int\zeta(x)F(x,x/\varepsilon)\,d\mu_\varepsilon(x)$. Periodicity and the uniform bounds give uniform local mass bounds, so the convergence also holds for continuous compactly supported test functions.
\end{proof}

Let $u_k^\varepsilon$ denote the convex solution of \eqref{eq:weighted-oscillatory-problem} with $(\mu,F)$ replaced by $(\mu_k,F_k)$ and boundary values $u_k$. Its effective solution is $u_k$. For each fixed $\varepsilon>0$, the right-hand sides are measures of uniformly bounded total variation on $\Omega$ and converge weakly there. The boundary values also converge uniformly. By the standard compactness argument (see \cite[Sections~2.2--2.4]{figalli2017monge} or \cite[Lemma~5.1 and its proof]{hartenstine2006dirichlet}), we obtain
\[
u_k^\varepsilon\to u^\varepsilon,\qquad u_k\to u\quad\text{uniformly on }\overline\Omega.
\]
Thus any estimate for $u_k^\varepsilon-u_k$ that is uniform in $k$ passes to the limit.

For the normalized periodic measures $g_i\mu$ considered in Section~\ref{sec:measure-valued-setting}, we use the common convolutions
\[
g_{i,k}=\frac{(g_i\mu)*\eta_k}{f_k}.
\]
These preserve normalization, the positive bounds, and inequalities between the densities; in particular,
\[
\|g_{1,k}-g_{2,k}\|_{L^\infty(d\mu_k)}\leq\|g_1-g_2\|_{L^\infty(d\mu)}.
\]

\section{Corrector expansions for H\"older densities}
\label{sec:corrector-estimates}

In this section we prove the following two rates for positive H\"older densities.

\begin{Theorem}
\label{thm:main-result}
Let $\mu=dy$, $m\in\{0,1\}$, $\alpha\in(0,1]$, and $\gamma\in(0,1)$. Assume that $F:\overline\Omega\times\mathbb T^n\to(0,\infty)$ satisfies $0<\lambda\leq F\leq\Lambda$. If $F\in C_x^{m,\alpha} (\overline\Omega;C_y^{0,\gamma}(\mathbb T^n))$ and $u\in C^{m+2,\alpha}(\overline\Omega)$, then
\[
\|u^\varepsilon-u\|_{L^\infty(\Omega)}\leq C\varepsilon^{m+\alpha},\qquad \forall\,0<\varepsilon\leq1.
\]
The constant $C$ depends only on $n$, $\alpha$, $\gamma$, $\lambda$, $\Lambda$, $\Omega$, and the indicated norms of $F$ and $u$.
\end{Theorem}

\begin{Remark}
For $m=0,1$, the assumption on $F$ in Theorem~\ref{thm:main-result} is satisfied under the joint regularity assumption
\[
F\in C^{m+\alpha+\gamma}(\overline\Omega\times\mathbb T^n),
\]
with the same $\gamma\in(0,1)$ as in the theorem. Here, we write $C^s=C^{k,\beta}$ for nonintegral $s=k+\beta$, where $k\in\mathbb N_0$ and $\beta\in(0,1)$.
\end{Remark}

To study the corrector expansion \eqref{eq:formal-third-order-expansion}, we first recall the following regularity result for periodic correctors; see \cite[Theorem~2.2]{li1990existence} and \cite[Theorem~0.2 and Remark~0.6]{caffarelli2004liouville}.

\begin{Theorem}
\label{thm:smooth-periodic-cell-problem}
Let $\gamma\in(0,1)$, $A\in\mathcal S_+^{n\times n}$, and let $\rho\in C^{0,\gamma}(\mathbb T^n)$ be positive with $\int_{\mathbb T^n}\rho=1$. Then $v_{A,\rho}\in C^{2,\gamma}(\mathbb T^n)$ and
\begin{equation}
\label{eq:parametric-corrector}
\det(A+D_y^2v_{A,\rho})=(\det A)\rho,\qquad A+D_y^2v_{A,\rho}>0\quad\text{in }\mathbb T^n.
\end{equation}
More generally, $\rho\in C^{k,\gamma}$ for an integer $k\geq0$ implies $v_{A,\rho}\in C^{k+2,\gamma}$, and smooth data give a smooth solution.
\end{Theorem}

We next study the dependence of $v_{A,\rho}$ on $(A,\rho)$. For $0<\gamma<1$ and $k\in\mathbb N_0$, define
\[
C_0^{k,\gamma}(\mathbb T^n):=\left\{f\in C^{k,\gamma}(\mathbb T^n):\int_{\mathbb T^n}f(y)\,dy=0\right\},
\]
and write
\[
X_\gamma:=\mathcal S_+^{n\times n}\times\left\{\rho\in C^{0,\gamma}(\mathbb T^n):\rho>0,\ \int_{\mathbb T^n}\rho(y)\,dy=1\right\}.
\]

\begin{Lemma}
\label{lem:smooth-cell-dependence}
The corrector map $(A,\rho)\mapsto v_{A,\rho}$ is $C^3$ from $X_\gamma$ to $C_0^{2,\gamma}(\mathbb T^n)$ in the Fr\'echet sense. If $A\in\mathcal K\Subset\mathcal S_+^{n\times n}$, $\rho\geq\lambda_0>0$, and $\|\rho\|_{C^{0,\gamma}}\leq M_0$, then
\[
\|v_{A,\rho}\|_{C^{2,\gamma}(\mathbb T^n)}+\sum_{j=1}^3\|D_{(A,\rho)}^jv_{A,\rho}\|_{C^{2,\gamma}(\mathbb T^n)}\leq C.
\]
Here parameter derivatives are measured in the induced operator norms, with matrix directions measured in the Euclidean norm and density directions in $C_0^{0,\gamma}(\mathbb T^n)$. The constant $C$ depends only on $n,\gamma,\mathcal K,\lambda_0,M_0$.
\end{Lemma}

\begin{proof}
The periodic $C^{2,\gamma}$ estimates and \eqref{eq:measure-corrector-size} give $\|v_{A,\rho}\|_{C^{2,\gamma}}\leq C$; see \cite[Remark~0.6]{caffarelli2004liouville}. Combining this bound with \eqref{eq:parametric-corrector} yields $cI\leq A+D_y^2v_{A,\rho}\leq CI$. The linearized Monge--Amp\`ere operator
\[
L_{A,\rho}\eta=\bigl[\operatorname{cof}(A+D_y^2v_{A,\rho})\bigr]^{ij}\eta_{ij}
\]
is self-adjoint because its coefficient matrix is divergence-free. By the strong maximum principle, its kernel on $C^{2,\gamma}$ consists exactly of constants. The Fredholm alternative and Schauder estimates therefore give an isomorphism $L_{A,\rho}:C_0^{2,\gamma}\to C_0^{0,\gamma}$ with a uniformly bounded inverse.

For $\zeta\in C_0^{2,\gamma}$, \eqref{eq:periodic-cell-compatibility} shows that $\det(A+D^2\zeta)-(\det A)\rho$ lies in $C_0^{0,\gamma}$. Applying the implicit-function theorem at $(A,\rho,v_{A,\rho})$, we obtain smooth dependence of $v_{A,\rho}$ on $(A,\rho)$. Differentiating \eqref{eq:parametric-corrector} with respect to $(A,\rho)$ up to three times, we obtain equations with the same operator $L_{A,\rho}$ and right-hand sides involving parameter directions and parameter derivatives of lower order. We then use the uniform inverse bound and the H\"older product estimate to estimate these derivatives successively.
\end{proof}

\begin{proof}[Proof of Theorem~\ref{thm:main-result}]
We divide the proof into several steps.

\smallskip

\noindent \textbf{Step 1.} It suffices to consider sufficiently small $\varepsilon$. Indeed, the Alexandrov maximum principle gives a uniform bound for $u^\varepsilon$, which yields the desired estimate for $\varepsilon$ bounded away from zero after increasing $C$. Set
\[
A(x):=D^2u(x),\qquad\rho(x,y):=\frac{F(x,y)}{\overline F(x)},\qquad Z(x):=(A(x),\rho(x,\cdot))\in X_\gamma.
\]
The bounds on $F$ give $\lambda\leq\overline F\leq\Lambda$, $c\leq\rho\leq C$, and $\int_{\mathbb T^n}\rho(x,y)\,dy=1$. Since $u$ is convex and $D^2u$ is bounded, the lower bound $\det D^2u=\overline F\geq\lambda$ gives
\begin{equation*}
cI\leq A(x)\leq CI\qquad\forall\,x\in\overline\Omega.
\end{equation*}
Here $c,C>0$ depend only on $n,\lambda,\Lambda$, and $\|u\|_{C^2(\overline\Omega)}$. Let $E_\gamma:=\mathcal S^{n\times n}\times C^{0,\gamma}(\mathbb T^n)$, equipped with the product norm $\|(B,f)\|_{E_\gamma}:=|B|+\|f\|_{C^{0,\gamma}(\mathbb T^n)}$. Then
\[
\|Z\|_{L^\infty(\Omega;E_\gamma)}=\sup_{x\in\Omega}\left(|A(x)|+\|\rho(x,\cdot)\|_{C^{0,\gamma}(\mathbb T^n)}\right).
\]
For derivatives of parameter fields, we apply the $E_\gamma$ norm componentwise. Since $\overline F\geq\lambda$, the normalization $\rho=F/\overline F$ preserves the $C^{m,\alpha}$ regularity in $x$, with values in $C^{0,\gamma}(\mathbb T^n)$, for $m=0,1$.

Using a bounded Whitney extension operator, we extend $A$ and the $C_0^{0,\gamma}(\mathbb T^n)$-valued field $\rho-1$ in $C^{m,\alpha}$ to a neighborhood of $\overline\Omega$, so that $\rho$ retains its normalization. After shrinking the neighborhood, $A$ and $\rho$ remain uniformly positive. We then mollify in $x$ with a standard nonnegative symmetric mollifier $\eta_\varepsilon(x)$ at scale $\varepsilon$ and set
\[
A_\varepsilon:=A*\eta_\varepsilon,\qquad\rho_\varepsilon:=\rho*\eta_\varepsilon,\qquad Z_\varepsilon:=(A_\varepsilon,\rho_\varepsilon)=Z*\eta_\varepsilon.
\]
For small $\varepsilon$, we have $Z_\varepsilon:\overline\Omega\to X_\gamma$ and
\[
cI\leq A_\varepsilon\leq CI,\qquad c\leq\rho_\varepsilon\leq C.
\]
We then denote the corresponding full solution on $\mathbb R^n$ and its periodic part by
\[
w^\varepsilon(x,y):=w_{A_\varepsilon(x),\rho_\varepsilon(x,\cdot)}(y),\qquad v^\varepsilon(x,y):=v_{A_\varepsilon(x),\rho_\varepsilon(x,\cdot)}(y).
\]

\smallskip

\noindent \textbf{Step 2.} Since $Z$ is $C^{0,\alpha}$ in $x$, standard mollification estimates give
\begin{equation}
\label{eq:mollified-hessian}
\|Z_\varepsilon-Z\|_{L^\infty(\Omega;E_\gamma)}+\varepsilon\|DZ_\varepsilon\|_{L^\infty(\Omega;E_\gamma)}+\varepsilon^2\|D^2Z_\varepsilon\|_{L^\infty(\Omega;E_\gamma)}\leq C\varepsilon^\alpha.
\end{equation}

Set
\[
V^\varepsilon(x):=u(x)+\varepsilon^2v^\varepsilon(x,x/\varepsilon).
\]
Differentiating $V^\varepsilon$ twice and writing $y=x/\varepsilon$, we obtain
\[
D_{ij}V^\varepsilon=A_{ij}+(v^\varepsilon)_{y_i y_j}+\varepsilon\bigl((v^\varepsilon)_{x_i y_j}+(v^\varepsilon)_{x_j y_i}\bigr)+\varepsilon^2(v^\varepsilon)_{x_i x_j}.
\]
Since $D_y^2w^\varepsilon=A_\varepsilon+D_y^2v^\varepsilon$, we have
\[
D^2V^\varepsilon(x)=D_y^2w^\varepsilon(x,x/\varepsilon)+E_\varepsilon(x),
\]
with
\[
\begin{aligned}
\|E_\varepsilon\|_{L^\infty(\Omega)}&\leq\|A-A_\varepsilon\|_{L^\infty(\Omega)}+C\varepsilon\|D_xD_yv^\varepsilon\|_{L^\infty(\Omega\times\mathbb T^n)}+\varepsilon^2\|D_x^2v^\varepsilon\|_{L^\infty(\Omega\times\mathbb T^n)}\\
&\leq C\bigl(\varepsilon^\alpha+\varepsilon^{2\alpha}\bigr)\\
&\leq C\varepsilon^\alpha,
\end{aligned}
\]
where we used \eqref{eq:mollified-hessian} and the estimates
\[
\begin{aligned}
\|D_xD_yv^\varepsilon\|_{L^\infty(\Omega\times\mathbb T^n)}&\leq\|D_xv^\varepsilon\|_{L^\infty(\Omega;C^{2,\gamma}(\mathbb T^n))}\leq C\|DZ_\varepsilon\|_{L^\infty(\Omega;E_\gamma)},\\
\|D_x^2v^\varepsilon\|_{L^\infty(\Omega\times\mathbb T^n)}&\leq C\left(\|D^2Z_\varepsilon\|_{L^\infty(\Omega;E_\gamma)}+\|DZ_\varepsilon\|_{L^\infty(\Omega;E_\gamma)}^{\,2}\right),
\end{aligned}
\]
which follow from Lemma~\ref{lem:smooth-cell-dependence} and the chain rule.

In conclusion, for sufficiently small $\varepsilon$, we have
\[
\|V^\varepsilon-u\|_{L^\infty(\overline\Omega)}\leq C\varepsilon^2,\qquad cI\leq D^2V^\varepsilon\leq CI.
\]
By the cell equation \eqref{eq:parametric-corrector},
\[
\det D_y^2w^\varepsilon(x,x/\varepsilon)=(\det A_\varepsilon(x))\rho_\varepsilon(x,x/\varepsilon).
\]
Recalling that $F=(\det A)\rho$, we use the bound on $E_\varepsilon$, \eqref{eq:mollified-hessian}, and the local Lipschitz continuity of the determinant to obtain
\[
\left|\det D^2V^\varepsilon-F(x,x/\varepsilon)\right|\leq C_0\varepsilon^\alpha.
\]
Since $\varepsilon^2\leq\varepsilon^\alpha$ for $0<\varepsilon\leq1$, Proposition~\ref{prop:approximate-solution-principle}, applied with $\beta=\alpha$, proves the case $m=0$.

\smallskip

\noindent  \textbf{Step 3.} We now prove the case $m=1$. For each $x\in\overline\Omega$, let $\psi^\varepsilon(x,\cdot)$ be the mean-zero solution of
\begin{equation}
\label{eq:third-corrector}
(W^\varepsilon)^{ij}(\psi^\varepsilon)_{y_i y_j}=-2(W^\varepsilon)^{ij}(v^\varepsilon)_{x_i y_j}\quad\text{in }\mathbb T^n,\qquad\int_{\mathbb T^n}\psi^\varepsilon(x,y)\,dy=0,
\end{equation}
where
\[
W^\varepsilon(x,y):=\operatorname{cof}D_y^2w^\varepsilon(x,y).
\]
Indeed, since $W^\varepsilon$ is divergence-free in $y$,
\[
\int_{\mathbb T^n}(W^\varepsilon)^{ij}(v^\varepsilon)_{x_i y_j}\,dy=0.
\]
The uniform ellipticity of $W^\varepsilon$ and the Fredholm alternative therefore give a unique mean-zero solution.

We next estimate the correctors and their derivatives in $x$. The $C^{1,\alpha}$ regularity of $Z$ gives
\begin{equation}
\label{eq:c1alpha-mollification}
\begin{alignedat}{2}
\|Z_\varepsilon-Z\|_{L^\infty(\Omega;E_\gamma)}&\leq C\varepsilon^{1+\alpha},\qquad&\|DZ_\varepsilon-DZ\|_{L^\infty(\Omega;E_\gamma)}&\leq C\varepsilon^\alpha,\\
\|D^2Z_\varepsilon\|_{L^\infty(\Omega;E_\gamma)}&\leq C\varepsilon^{\alpha-1},\qquad&\|D^3Z_\varepsilon\|_{L^\infty(\Omega;E_\gamma)}&\leq C\varepsilon^{\alpha-2}.
\end{alignedat}
\end{equation}
The first estimate follows from the symmetry of $\eta_\varepsilon$. Applying standard mollification estimates to $DZ\in C^{0,\alpha}$, we obtain the other three. Lemma~\ref{lem:smooth-cell-dependence} and the chain rule give
\[
\begin{alignedat}{2}
\|v^\varepsilon\|_{L^\infty(\Omega;C^{2,\gamma}(\mathbb T^n))}&\leq C,\qquad&\|D_xv^\varepsilon\|_{L^\infty(\Omega;C^{2,\gamma}(\mathbb T^n))}&\leq C,\\
\|D_x^2v^\varepsilon\|_{L^\infty(\Omega;C^{2,\gamma}(\mathbb T^n))}&\leq C\varepsilon^{\alpha-1},\qquad&\|D_x^3v^\varepsilon\|_{L^\infty(\Omega;C^{2,\gamma}(\mathbb T^n))}&\leq C\varepsilon^{\alpha-2}.
\end{alignedat}
\]
Since $W^\varepsilon=\operatorname{cof}D_y^2w^\varepsilon$, these estimates also give
\[
\|W^\varepsilon\|_{L^\infty(\Omega;C^{0,\gamma}(\mathbb T^n))}+\|D_xW^\varepsilon\|_{L^\infty(\Omega;C^{0,\gamma}(\mathbb T^n))}+\varepsilon^{1-\alpha}\|D_x^2W^\varepsilon\|_{L^\infty(\Omega;C^{0,\gamma}(\mathbb T^n))}\leq C.
\]
Applying Schauder estimates to \eqref{eq:third-corrector} and to the equations obtained by differentiating it once and twice with respect to $x$, we obtain
\[
\|\psi^\varepsilon\|_{L^\infty(\Omega;C^{2,\gamma}(\mathbb T^n))}+\varepsilon^{1-\alpha}\|D_x\psi^\varepsilon\|_{L^\infty(\Omega;C^{2,\gamma}(\mathbb T^n))}+\varepsilon^{2-\alpha}\|D_x^2\psi^\varepsilon\|_{L^\infty(\Omega;C^{2,\gamma}(\mathbb T^n))}\leq C.
\]

Set
\[
\widetilde V^\varepsilon(x):=u(x)+\varepsilon^2v^\varepsilon(x,x/\varepsilon)+\varepsilon^3\psi^\varepsilon(x,x/\varepsilon).
\]
Differentiating twice, we obtain
\begin{equation}
\label{eq:third-order-hessian}
D^2\widetilde V^\varepsilon(x)=D_y^2w^\varepsilon(x,x/\varepsilon)+\varepsilon N_\varepsilon(x,x/\varepsilon)+R_\varepsilon(x,x/\varepsilon),
\end{equation}
where
\[
(N_\varepsilon)_{ij}(x,y):=(v^\varepsilon)_{x_i y_j}+(v^\varepsilon)_{x_j y_i}+(\psi^\varepsilon)_{y_i y_j}
\]
and the remainder term is given by
\[
(R_\varepsilon)_{ij}(x,y):=(A-A_\varepsilon)_{ij}+\varepsilon^2(v^\varepsilon)_{x_i x_j}+\varepsilon^2\bigl((\psi^\varepsilon)_{x_i y_j}+(\psi^\varepsilon)_{x_j y_i}\bigr)+\varepsilon^3(\psi^\varepsilon)_{x_i x_j}.
\]
Substituting the preceding derivative bounds, we obtain
\[
\|R_\varepsilon\|_{L^\infty(\Omega\times\mathbb T^n)}\leq \|A-A_\varepsilon\|_{L^\infty(\Omega)}+C\bigl(\varepsilon^2\varepsilon^{\alpha-1}+\varepsilon^3\varepsilon^{\alpha-2}\bigr)\leq C\varepsilon^{1+\alpha}.
\]
Moreover, $N_\varepsilon$ and $D_y^2w^\varepsilon$ are uniformly bounded, and $D_y^2w^\varepsilon$ is uniformly positive definite. Together with \eqref{eq:third-order-hessian} and the bound on $R_\varepsilon$, this gives
\[
cI\leq D^2\widetilde V^\varepsilon\leq CI
\]
for sufficiently small $\varepsilon$. By \eqref{eq:third-corrector},
\[
(W^\varepsilon)^{ij}(N_\varepsilon)_{ij}=0\qquad\text{in }\Omega\times\mathbb T^n.
\]
Thus the term of order $\varepsilon$ in the determinant expansion vanishes. Using the cell equation \eqref{eq:parametric-corrector}, \eqref{eq:c1alpha-mollification}, and the bound on $R_\varepsilon$, we obtain
\[
\begin{aligned}
\left|\det D^2\widetilde V^\varepsilon-F(x,x/\varepsilon)\right|&\leq C\bigl(\varepsilon^2+\varepsilon^{1+\alpha}\bigr)+C\|(\det A_\varepsilon)\rho_\varepsilon-(\det A)\rho\|_{L^\infty(\Omega\times\mathbb T^n)}\\
&\leq C\varepsilon^{1+\alpha}.
\end{aligned}
\]
Finally, the uniform bounds for $v^\varepsilon$ and $\psi^\varepsilon$ give
\[
\|\widetilde V^\varepsilon-u\|_{L^\infty(\overline\Omega)}\leq C\varepsilon^2.
\]
Since $\varepsilon^2\leq\varepsilon^{1+\alpha}$ for $0<\varepsilon\leq1$, Proposition~\ref{prop:approximate-solution-principle} applies with $\beta=1+\alpha$ and completes the proof.
\end{proof}

\begin{Example}
\label{ex:fixed-data-sharpness}
Let $\Omega=(0,1)$, $\varphi=0$, and consider
\[
f(x):=a_0+\sum_{\substack{k\geq1\\
k\text{ even}}}\bigl(a_k\cos(2\pi kx)+b_k\sin(2\pi kx)\bigr),
\]
where the coefficients are absolutely summable and $f>0$. Set
\[
g(y):=1+\frac12\cos(2\pi y),\qquad F(x,y):=f(x)g(y).
\]
For each even integer $N\geq2$, let $u_N$ denote the oscillatory solution for $\varepsilon=N^{-1}$. Since $\overline F=f$, the oscillatory and effective solutions satisfy
\[
u_N''=f(x)g(Nx),\qquad u''=f(x)\quad\text{in }(0,1),\qquad u_N=u=0\quad\text{on }\{0,1\}.
\]
Multiplying the series for $f$ by $g(Nx)$, we obtain
\[
f(x)g(Nx)=a_0+\frac{a_N}{4}+\sum_{\substack{k\geq1\\
k\text{ even}}}\bigl(a_{N,k}\cos(2\pi kx)+b_{N,k}\sin(2\pi kx)\bigr),
\]
where, with $b_0:=0$,
\[
a_{N,k}=a_k+\frac14\bigl(a_{|k-N|}+a_{k+N}+a_0\delta_{kN}\bigr),\qquad b_{N,k}=b_k+\frac14\bigl(\operatorname{sgn}(k-N)b_{|k-N|}+b_{k+N}\bigr).
\]
Here $\delta_{kN}$ is the Kronecker delta. Integrating twice and imposing the boundary conditions, we obtain
\[
\begin{aligned}
u_N(x)&=\frac12\left(a_0+\frac{a_N}{4}\right)x(x-1)+\sum_{\substack{k\geq1\\
k\text{ even}}}\frac{a_{N,k}\bigl(1-\cos(2\pi kx)\bigr)-b_{N,k}\sin(2\pi kx)}{4\pi^2k^2},\\
u(x)&=\frac{a_0}{2}x(x-1)+\sum_{\substack{k\geq1\\
k\text{ even}}}\frac{a_k\bigl(1-\cos(2\pi kx)\bigr)-b_k\sin(2\pi kx)}{4\pi^2k^2}.
\end{aligned}
\]
At $x=1/2$, both sums vanish. Thus $u(1/2)=-a_0/8$ and
\begin{equation}
\label{eq:fourier-midpoint}
u_N\left(\frac12\right)=-\frac{a_0}{8}-\frac{a_N}{32},\qquad(u_N-u)\left(\frac12\right)=-\frac{a_N}{32}.
\end{equation}

Now fix $s\in(0,2)\setminus\{1\}$ and set $N_j:=2^{j+1}$. Take $a_0=1$, $b_k=0$, and
\[
a_k:=
\begin{cases}
\eta N_j^{-s},&k=N_j,\quad j\geq1,\\
0,&\text{otherwise}
\end{cases}
.
\]
Denote the resulting functions by $f_s$, $F_s$, $u_s$, and $u_{s,N}$, with $\eta>0$ chosen so that $f_s\geq1/2$. The density $F_s$ is independent of $\varepsilon$. A dyadic splitting shows that
\[
\begin{cases}
f_s\in C^{0,s}([0,1]),\quad u_s\in C^{2,s}([0,1]),&0<s<1,\\
f_s\in C^{1,s-1}([0,1]),\quad u_s\in C^{3,s-1}([0,1]),&1<s<2
\end{cases}
.
\]
For $\varepsilon_j:=N_j^{-1}$, \eqref{eq:fourier-midpoint} gives
\[
(u_{s,N_j}-u_s)\left(\frac12\right)=-\frac{\eta}{32}\varepsilon_j^s.
\]
Taking $s=\alpha$ and $s=1+\alpha$, respectively, proves the sharpness of the two rates in Theorem~\ref{thm:main-result} for $0<\alpha<1$. Here $F_s/\overline F_s=g$ is independent of $x$, so the only varying cell parameter is $A=u_s''=\overline F_s=f_s$.

For $s=1,2$, replace $a_{N_j}$ by $\eta j^{-2}N_j^{-s}$. Then $f_s\in C^{s-1,1}([0,1])$ and $u_s\in C^{s+1,1}([0,1])$. Applying \eqref{eq:fourier-midpoint} again gives
\[
\left|(u_{s,N_j}-u_s)\left(\frac12\right)\right|=\frac{\eta}{32j^2}\varepsilon_j^s.
\]
Since $j^{-2}\varepsilon_j^{-\theta}\to\infty$ for every $\theta>0$, neither endpoint admits a larger algebraic exponent.
\end{Example}

\section{Lower-order estimates for general periodic measures}
\label{sec:measure-valued-setting}

In this section we prove the case $m=0$ of Theorem~\ref{thm:weighted-measure-holder-rate}. We first recall an almost H\"older estimate for the linearized equation, which we use to prove Lipschitz dependence of periodic correctors on the Hessian and the microscopic weight.

Throughout this section, $\mu$ denotes a normalized nonnegative locally finite periodic Borel measure. We use the notation $v_A$ and $w_A$ from Section~\ref{sec:periodic-cell-problems}.

\subsection{An estimate for the linearized equation}
\label{sec:linearized-periodic-estimate}

Let $\Phi$ be a convex function, $x_0$ a point in its domain, and $p\in\partial\Phi(x_0)$. For $t>0$, the section of $\Phi$ at $x_0$ of height $t$ is defined by
\[
S_{t,p}^{\Phi}(x_0):=\left\{x:\Phi(x)<\Phi(x_0)+p\cdot(x-x_0)+t\right\}.
\]
For a lattice $\Gamma$ with fundamental domain $D(\Gamma)$, we measure the relative size of $\Omega$ by
\[
\kappa_\Gamma(\Omega):=\sup\{t>0:y+tD(\Gamma)\subset\Omega\text{ for some }y\in\mathbb R^n\}.
\]

For a smooth potential $\phi$ with $D^2\phi>0$, the linearized Monge--Amp\`ere operator is
\[
L_\phi f:=\bigl[\operatorname{cof}(D^2\phi)\bigr]^{ij}f_{ij}.
\]

The next proposition follows from \cite[Propositions~2.4 and~2.17]{jin2025liouville}. When $\kappa_\Gamma(\Omega)$ is sufficiently large, \cite[Proposition~2.4]{jin2025liouville} gives two-sided ball comparability for the interior sections, with constants depending only on $n$. These are the only geometric bounds used in the proof of \cite[Proposition~2.17]{jin2025liouville}, so the same argument applies.

\begin{Proposition}
\label{prop:jltt-paired-holder}
Let $\Gamma$ be a lattice with fundamental domain $D(\Gamma)$, and let $f$ be a smooth positive $\Gamma$-periodic function satisfying
\[
\int_{D(\Gamma)}f(y)\,dy=|D(\Gamma)|.
\]
For $i=1,2$, let $\Omega_i$ be a bounded convex domain with $B_1\subset\Omega_i\subset B_n$, and let $\phi_i\in C^\infty(\Omega_i)\cap C(\overline{\Omega_i})$ satisfy $D^2\phi_i>0$ and
\[
\det D^2\phi_i=f\quad\text{in }\Omega_i,\qquad\phi_i=0\quad\text{on }\partial\Omega_i.
\]
Let $\Lambda_*=\Lambda_*(n)>1$ be sufficiently large, and assume that $\kappa_\Gamma(\Omega_i)\geq\Lambda_*^2$ for $i=1,2$. Let $w\in C^2(\Omega_1\cup\Omega_2)$ satisfy
\[
L_{\phi_1}w\leq0\quad\text{in }\Omega_1,\qquad L_{\phi_2}w\geq0\quad\text{in }\Omega_2.
\]
Then there exist $\alpha_0\in(0,1)$ and $C>0$, depending only on $n$, such that
\begin{equation*}
|w(x)-w(y)|\leq C\left(|x-y|^{\alpha_0}+\kappa^{-\alpha_0/4}\right)\|w\|_{L^\infty(B_{1/2})},\qquad x,y\in B_{1/4},
\end{equation*}
where $\kappa:=\kappa_\Gamma(B_1)$.
\end{Proposition}

\subsection{Stability of periodic correctors}
\label{sec:corrector-parameter-stability}

\begin{Theorem}
\label{thm:measure-corrector-lipschitz}
Let $\mathcal K\Subset\mathcal S_+^{n\times n}$, and fix $0<\lambda_0\leq\Lambda_0$. Let $\mu$ be a normalized nonnegative locally finite $\mathbb Z^n$-periodic Borel measure. For $i=1,2$, let $A_i\in\mathcal K$ and let $g_i$ be a periodic Borel function satisfying
\[
\lambda_0\leq g_i\leq\Lambda_0\quad\mu\text{-a.e.},\qquad\int_{\mathbb T^n}g_i\,d\mu=1.
\]
Then
\begin{equation*}
\left\|v_{A_1,g_1\mu}-v_{A_2,g_2\mu}\right\|_{L^\infty(\mathbb T^n)}\leq C\left(|A_1-A_2|+\|g_1-g_2\|_{L^\infty(d\mu)}\right).
\end{equation*}
Here $C$ depends only on $n$, $\mathcal K$, $\lambda_0$, and $\Lambda_0$.
\end{Theorem}

\begin{proof}
By the approximation in Section~\ref{sec:measure-approximation}, we may assume that $d\mu=f\,dy$, with $f,g_1,g_2$ smooth and positive. In dimension one, the cell equation is linear and the estimate follows by direct integration; henceforth we assume that $n\geq2$.

\smallskip

\noindent \textbf{Step 1.} We first show that
\begin{equation}
\label{eq:measure-corrector-lipschitz}
\left\|v_{A_1,\mu}-v_{A_2,\mu}\right\|_{L^\infty(\mathbb T^n)}\leq C_{\mathcal K}|A_1-A_2|.
\end{equation}
Throughout this step, all constants, as well as the lower bounds on $R$, depend only on $n$ and $\mathcal K$. For simplicity, we assume that $\det A_1=\det A_2=1$. The general case follows by applying the estimate to $(\det A_i)^{-1/n}A_i$ and using $v_{tA,\mu}=t\,v_{A,\mu}$ together with \eqref{eq:measure-corrector-size}. Moreover, in view of \eqref{eq:measure-corrector-size}, it suffices to consider the case $|A_1-A_2|<\delta_0$, where we choose $\delta_0=\delta_0(n,\mathcal K)>0$ sufficiently small so that
\begin{equation}
\label{eq:close-matrix-comparability}
\frac{3}{4}I\leq A_1^{-1/2}A_2A_1^{-1/2}\leq\frac{5}{4}I.
\end{equation}

Let us write
\[
v_i:=v_{A_i,\mu},\qquad w_i:=w_{A_i,\mu}=q_{A_i}+v_i,\qquad M:=\operatorname{osc}_{\mathbb T^n}(v_1-v_2).
\]
Let $\alpha_0$ be the exponent in Proposition~\ref{prop:jltt-paired-holder} and let $C_{\mathcal K}$ be sufficiently large. We claim that there exists $R_0=R_0(n,\mathcal K)$ such that, for every $R\geq R_0$,
\begin{equation}
\label{eq:corrector-almost-holder}
|(w_1-w_2)(x)-(w_1-w_2)(\tilde x)|\leq\theta_R\left(M+C_{\mathcal K}|A_1-A_2|R^2\right),\qquad x,\tilde x\in D_0,
\end{equation}
where
\begin{equation*}
D_0:=[-1/2,1/2]^n,\qquad\theta_R:=C_{\mathcal K}R^{-\alpha_0/4},\qquad R\geq1.
\end{equation*}
Once \eqref{eq:corrector-almost-holder} is established, \eqref{eq:measure-corrector-lipschitz} follows. Indeed, let $x_+,x_-\in D_0$ be points where $v_1-v_2$ attains its maximum and minimum, respectively. Applying \eqref{eq:corrector-almost-holder} at these points with $R=R(n,\mathcal K)\geq R_0$ chosen so that $\theta_R\leq1/2$, and using $|q_{A_1-A_2}(x_+)-q_{A_1-A_2}(x_-)|\leq C_n|A_1-A_2|$, we obtain
\[
M\leq C_{\mathcal K}|A_1-A_2|+\frac12M.
\]
Since $v_1-v_2$ has mean zero, this gives $\|v_1-v_2\|_{L^\infty(\mathbb T^n)}\leq M\leq C_{\mathcal K}|A_1-A_2|$.

It remains to prove \eqref{eq:corrector-almost-holder}. For $R\geq R_0$, set
\[
T:=A_1^{-1/2},\qquad w_i^R(x):=R^{-2}w_i(RTx),\qquad\Omega_i^R:=\{x:w_i^R(x)<5/4\}.
\]
By \eqref{eq:measure-corrector-size},
\[
w_i^R(x)=\frac12x\cdot TA_iTx+O_{\mathcal K}(R^{-2}),\qquad x\in\mathbb R^n,\quad i=1,2.
\]
For the lattice $\Gamma_R:=R^{-1}A_1^{1/2}\mathbb Z^n$, we have $\kappa_{\Gamma_R}(B_1)\simeq_{\mathcal K}R$. Combining these estimates with the bounds $\frac34I\leq TA_iT\leq\frac54I$ from \eqref{eq:close-matrix-comparability}, we obtain, for sufficiently large $R_0$ and all $R\geq R_0$,
\[
B_1\subset\Omega_i^R\subset B_n,\qquad\kappa_{\Gamma_R}(\Omega_i^R)\geq\kappa_{\Gamma_R}(B_1)\geq c_{\mathcal K}R\geq\Lambda_*^2,\qquad i=1,2.
\]

Since $\det A_i=\det T=1$, the functions $w_i^R$ satisfy
\[
\det D^2w_i^R=f(RTx)\quad\text{in }\Omega_i^R,\qquad w_i^R=5/4\quad\text{on }\partial\Omega_i^R,\qquad i=1,2.
\]
We note that
\[
w_1^R(x)-w_2^R(x)=\frac12x\cdot T(A_1-A_2)Tx+R^{-2}(v_1-v_2)(RTx),
\]
and recalling that $M=\operatorname{osc}_{\mathbb T^n}(v_1-v_2)$, we obtain
\[
\|w_1^R-w_2^R\|_{L^\infty(B_{1/2})}\leq C_{\mathcal K}|A_1-A_2|+\frac{M}{R^2}.
\]

By the concavity of $N\mapsto(\det N)^{1/n}$, we have
\[
L_{w_1^R}(w_1^R-w_2^R)\leq0\quad\text{in }\Omega_1^R,\qquad L_{w_2^R}(w_1^R-w_2^R)\geq0\quad\text{in }\Omega_2^R.
\]
Fix $x,\tilde x\in D_0$ and set
\[
x_R:=R^{-1}T^{-1}x,\qquad\tilde x_R:=R^{-1}T^{-1}\tilde x.
\]
As $R_0$ is sufficiently large, we have $x_R,\tilde x_R\in B_{1/4}$ and $|x_R-\tilde x_R|\leq C_{\mathcal K}R^{-1}$. After subtracting the boundary constants from the potentials, we apply Proposition~\ref{prop:jltt-paired-holder} to $w_1^R-w_2^R$ with $\kappa=\kappa_{\Gamma_R}(B_1)\simeq_{\mathcal K}R$. Using the $L^\infty(B_{1/2})$ estimate above, we obtain
\[
\begin{aligned}
\frac1{R^2}|(w_1-w_2)(x)-(w_1-w_2)(\tilde x)|&=|(w_1^R-w_2^R)(x_R)-(w_1^R-w_2^R)(\tilde x_R)|\\
&\leq\theta_R\left(C_{\mathcal K}|A_1-A_2|+\frac{M}{R^2}\right),
\end{aligned}
\]
which proves \eqref{eq:corrector-almost-holder}.

\smallskip
\noindent\textbf{Step 2.} We now fix $A_0\in\mathcal K$ and prove that
\begin{equation}
\label{eq:weighted-profile-lipschitz}
\left\|v_{A_0,g_1\mu}-v_{A_0,g_2\mu}\right\|_{L^\infty(\mathbb T^n)}\leq C\|g_1-g_2\|_{L^\infty(d\mu)},
\end{equation}
with $C$ independent of $A_0$. For simplicity, we again assume that $\det A_0=1$. In view of \eqref{eq:measure-corrector-size}, we may assume that $\eta:=\lambda_0^{-1}\|g_1-g_2\|_{L^\infty(d\mu)}<1/2$. Since $g_2\geq\lambda_0$, we then have
\begin{equation}
\label{eq:weighted-profile-measure-order}
(1-\eta)g_2\mu\leq g_1\mu\leq(1+\eta)g_2\mu.
\end{equation}

Write $f_i:=g_i f$ and
\[
v_i:=v_{A_0,f_i},\qquad w_i:=w_{A_0,f_i}=q_{A_0}+v_i,\qquad M:=\operatorname{osc}_{\mathbb T^n}(v_1-v_2).
\]
We set
\[
\begin{alignedat}{2}
T&:=A_0^{-1/2},\qquad&\Gamma_R&:=R^{-1}A_0^{1/2}\mathbb Z^n,\\
w_i^R(x)&:=R^{-2}w_i(RTx),\qquad&f_{i,R}(x)&:=f_i(RTx),
\end{alignedat}
\]
and define
\[
\Omega_1:=\{x:w_1^R(x)<5/4\},\qquad\Omega_2:=\{x:w_2^R(x)<3/2\}.
\]
As in Step~1, for sufficiently large $R$, we have
\[
B_1\subset\Omega_1\Subset\Omega_2\subset B_n,\qquad\kappa_{\Gamma_R}(\Omega_i)\geq\kappa_{\Gamma_R}(B_1)\geq c_{\mathcal K}R\geq\Lambda_*^2,\qquad i=1,2,
\]
and $w_i^R$ satisfy $\det D^2w_i^R=f_{i,R}$.

Let $\chi$ be the convex solution of
\[
\det D^2\chi=f_{1,R}\quad\text{in }\Omega_2,\qquad\chi=3/2\quad\text{on }\partial\Omega_2.
\]
The classical interior regularity theory for the Monge--Amp\`ere equation \cite{caffarelli1990interior} implies that $\chi\in C^\infty(\Omega_2)$ and $D^2\chi>0$ in $\Omega_2$. By \eqref{eq:weighted-profile-measure-order}, we have $(1-\eta)f_{2,R}\leq f_{1,R}\leq(1+\eta)f_{2,R}$. The comparison principle therefore gives
\[
\frac32+(1+\eta)^{1/n}\bigl(w_2^R-\tfrac32\bigr)\leq\chi\leq\frac32+(1-\eta)^{1/n}\bigl(w_2^R-\tfrac32\bigr).
\]
Using $|(1\pm\eta)^{1/n}-1|\leq C\eta$ and $-2\leq w_2^R-3/2\leq0$ in $\Omega_2$, we obtain
\begin{equation}
\label{eq:weighted-dirichlet-replacement}
\|\chi-w_2^R\|_{L^\infty(\Omega_2)}\leq C\eta.
\end{equation}
Note that
\[
w_1^R(x)-w_2^R(x)=R^{-2}(v_1-v_2)(RTx).
\]
Together with \eqref{eq:weighted-dirichlet-replacement}, this gives
\[
\|w_1^R-\chi\|_{L^\infty(B_{1/2})}\leq C\eta+\frac{M}{R^2}.
\]

The functions $w_1^R$ and $\chi$ satisfy the same equation in $\Omega_2$. By the concavity of $N\mapsto(\det N)^{1/n}$, we have
\[
L_{w_1^R}(w_1^R-\chi)\leq0\quad\text{in }\Omega_1,\qquad L_{\chi}(w_1^R-\chi)\geq0\quad\text{in }\Omega_2.
\]
We choose $x_+,x_-\in D_0$ where $v_1-v_2$ attains its maximum and minimum, respectively, and set $\tilde x_\pm:=R^{-1}T^{-1}x_\pm$. We still have $\tilde x_\pm\in B_{1/4}$ and $|\tilde x_+-\tilde x_-|\leq C_{\mathcal K}R^{-1}$, and by \eqref{eq:weighted-dirichlet-replacement},
\[
\begin{aligned}
\frac{M}{R^2}&=(w_1^R-w_2^R)(\tilde x_+)-(w_1^R-w_2^R)(\tilde x_-)\\
&\leq |(w_1^R-\chi)(\tilde x_+)-(w_1^R-\chi)(\tilde x_-)|+C\eta.
\end{aligned}
\]
After subtracting the boundary constants from the potentials, we apply Proposition~\ref{prop:jltt-paired-holder} to $w_1^R-\chi$ with $\kappa=\kappa_{\Gamma_R}(B_1)\simeq_{\mathcal K}R$. The preceding estimates then imply, for a fixed $R=R(n,\mathcal K)$ with $\theta_R\leq1/2$, that
\[
M\leq C\eta+\frac12M.
\]
Since $v_1-v_2$ has mean zero, we conclude that
\[
\|v_1-v_2\|_{L^\infty(\mathbb T^n)}\leq M\leq C\eta=\frac{C}{\lambda_0}\|g_1-g_2\|_{L^\infty(d\mu)},
\]
which proves \eqref{eq:weighted-profile-lipschitz} and completes the proof.
\end{proof}

\subsection{Proof of the lower-order estimate}
\label{sec:measure-rate-proof}

\begin{proof}[Proof of the case $m=0$ of Theorem~\ref{thm:weighted-measure-holder-rate}]
By the approximation in Section~\ref{sec:measure-approximation}, we may assume that $d\mu=f\,dy$, with $f$ smooth, positive, and periodic, that $F$ is smooth in all variables, and that $u$ is smooth on a neighborhood of $\overline\Omega$. After a fixed translation and dilation, we may also assume that $\overline\Omega\subset B_1(0)$. By periodicity, the normalization $\mu(\mathbb T^n)=1$, and the bound $F\leq\Lambda$, we have
\[
\int_\Omega F(x,x/\varepsilon)\,d\mu_\varepsilon(x)\leq C,\qquad 0<\varepsilon\leq1.
\]
The Alexandrov maximum principle then gives a uniform bound for $u^\varepsilon$. It therefore suffices to consider $0<\varepsilon<\varepsilon_0$; the remaining values of $\varepsilon$ are covered by increasing $C$.

Let $L:=\sup_{y\in\mathbb T^n} [F(\cdot,y)]_{C^{0,\alpha}(\overline\Omega)}$. Let us set
\[
A(x):=D^2u(x),\qquad g_x(y):=\frac{F(x,y)}{\overline F(x)}.
\]
Then $g_x\mu$ is a normalized periodic measure for each $x\in\overline\Omega$, and we have
\begin{equation}
\label{eq:weighted-normalized-profile}
\begin{gathered}
\lambda\leq \overline F\leq\Lambda,\qquad\frac{\lambda}{\Lambda}\leq g_x\leq\frac{\Lambda}{\lambda}\quad\mu\text{-a.e.},\\
|\overline F(x)-\overline F(\tilde x)|+\|g_x-g_{\tilde x}\|_{L^\infty(d\mu)}\leq C|x-\tilde x|^\alpha.
\end{gathered}
\end{equation}
Since $\det A=\overline F\geq\lambda$ and $D^2u$ is bounded, $A(\overline\Omega)$ is contained in a compact subset of $\mathcal S_+^{n\times n}$ depending only on the stated data. By Theorem~\ref{thm:measure-corrector-lipschitz}, \eqref{eq:weighted-normalized-profile}, and \eqref{eq:measure-corrector-size}, we have
\begin{equation}
\label{eq:weighted-modulated-corrector-bounds}
\sup_{x\in\overline\Omega}\|v_{A(x),g_x\mu}\|_{L^\infty(\mathbb T^n)}\leq C,\qquad\left\|v_{A_1,g_x\mu}-v_{A_2,g_{\tilde x}\mu}\right\|_{L^\infty(\mathbb T^n)}\leq C\left(|A_1-A_2|+|x-\tilde x|^\alpha\right)
\end{equation}
for $A_1,A_2$ in a fixed compact neighborhood of $A(\overline\Omega)$.

We now consider the second-order approximation
\[
V_F^\varepsilon(x):=u(x)+\varepsilon^2v_{A(x),g_x\mu}(x/\varepsilon).
\]
By \eqref{eq:weighted-modulated-corrector-bounds},
\[
\|V_F^\varepsilon-u\|_{L^\infty(\Omega)}\leq C\varepsilon^2.
\]
Fix $x_*\in\Omega$, take $r=\varepsilon$, and set $\ell_{x_*}(x)=u(x_*)+Du(x_*)\cdot(x-x_*)$. Define the convex function
\[
W_{x_*}(x)=\ell_{x_*}(x)+\frac12(x-x_*)\cdot A(x_*)(x-x_*)+\varepsilon^2v_{A(x_*),g_{x_*}\mu}(x/\varepsilon).
\]
By Taylor's formula and \eqref{eq:weighted-modulated-corrector-bounds}, we have
\[
|W_{x_*}-\ell_{x_*}|\leq Cr^2,\qquad|W_{x_*}-V_F^\varepsilon|\leq C\varepsilon^\alpha r^2
\]
in $\Omega\cap B_r(x_*)$. Moreover, $W_{x_*}$ satisfies $\det D^2W_{x_*}=F(x_*,x/\varepsilon)\mu_\varepsilon$. Since $F\geq\lambda>0$, the H\"older bound on $F$ gives
\[
(1-C\varepsilon^\alpha)\det D^2u^\varepsilon\leq\det D^2W_{x_*}\leq(1+C\varepsilon^\alpha)\det D^2u^\varepsilon
\]
in $\Omega\cap B_r(x_*)$. Applying Proposition~\ref{prop:local-comparison} with $\beta=\alpha$ and $V^\varepsilon=V_F^\varepsilon$, we conclude that
\[
\|u^\varepsilon-u\|_{L^\infty(\Omega)}\leq C\varepsilon^\alpha.
\]
\end{proof}

\section{Higher-order estimates for general periodic measures}
\label{sec:higher-measure-rate}

We prove the case $m=1$ of Theorem~\ref{thm:weighted-measure-holder-rate}. By Section~\ref{sec:measure-approximation}, we first work with smooth positive approximations, retaining the notation $u,F,\mu$. Thus $\mu$ has a smooth positive periodic density, $u\in C^\infty(\mathcal O)$, and $F\in C^\infty(\mathcal O\times\mathbb T^n)$, where $\mathcal O$ is a fixed neighborhood of $\overline\Omega$. All quantitative estimates and the threshold $\varepsilon_0$ depend only on the quantities specified in the theorem, independently of the upper and lower bounds and derivatives of the approximating periodic density and higher regularity norms of the data.

\subsection{Cell averages and reconstruction kernels}
\label{sec5:cell-reconstruction}

Fix a smooth mollifier
\[
\eta(y):=c_n
\begin{cases}
\exp\!\left(-\dfrac{1}{1-(100n)^2|y|^2}\right),& |y|<\dfrac{1}{100n},\\
0,& |y|\geq\dfrac{1}{100n},
\end{cases}
\]
where $c_n>0$ is chosen so that $\int_{\mathbb R^n}\eta(y)\,dy=1$, and set
\[
Q_z=z+\left[-\frac12,\frac12\right)^n,\qquad\chi_z=\mathbf1_{Q_z}*\eta.
\]

Let $d\mu_0=f\,dx$ be a smooth positive $\mathbb Z^n$-periodic measure satisfying
\begin{equation}
\label{eq:reference-cell-mass}
\int_{\mathbb T^n}f(y)\,dy=1,\qquad\int_{[-\frac12+\frac{3}{100n},\,\frac12-\frac{3}{100n}]^n}f(y)\,dy\geq\frac34,
\end{equation}
where these conditions can always be achieved by a translation and normalization. We define the associated cell averaging operator $\Pi:=\Pi_{\mu_0}$ on locally integrable functions $w$ on $\mathbb R^n$ by
\begin{equation*}
\Pi[w](z):=\langle w\rangle(z)=\int_{\mathbb R^n}\chi_z w f\,dx,\quad\forall\, z\in\mathbb Z^n,
\end{equation*}
and denote the cell barycenters with respect to $\mu_0$ by
\begin{equation}
\label{eq:barycenterslattice}
\bar x_z:=\langle x\rangle(z)=\int_{\mathbb R^n}\chi_z x f\,dx=z+\bar x_0.
\end{equation}
Here $\bar x_0:=\int_{\mathbb R^n}\chi_0x f\,dx$ is the barycenter of the zeroth cell. Hence, $|\xi-\bar x_z|\le|\xi-z|+|\bar x_z-z| \le C$ for every $\xi\in\operatorname{supp}\chi_z$. The formal transpose of $\Pi$ is given by
\[
\Pi^T[a](x):=\left(\sum_{z\in\mathbb Z^n}a(z)\chi_z(x)\right)f(x),\quad\forall\,  a:\mathbb Z^n\to\mathbb R.
\]
 
The weights satisfy $\chi_z(x)=\chi_0(x-z)$ and $\sum_z\chi_z(x)=1$. Hence, for every $\mathbb Z^n$-periodic function $w$, we have
\[
\Pi[w](z)=\int_{\mathbb T^n}w f\,dx.
\]

The constraint $\Pi[\widetilde u-u]=0$ gives the value estimate needed for Proposition~\ref{prop:local-comparison}. At the microscopic scale, the bound below is $O(1)$. After rescaling, this gives an $O(\varepsilon^2)$ error in the original variables.

\begin{Lemma}
\label{lem:cell-average-convex-bound}
Let $\widetilde u$ be a convex function on $\mathbb R^n$, and let $u$ be a smooth convex function satisfying $D^2u\leq\widetilde CI$ for some $\widetilde C>0$. If
\[
\Pi[\widetilde u-u]=0,
\]
then there exists $C>0$, depending only on $n$ and $\widetilde C$, such that
\[
\|\widetilde u-u\|_\infty\leq C,\qquad |D\widetilde u-Du|\leq C,
\]
where the gradient bound holds at every point of differentiability of $\widetilde u$.
\end{Lemma}

\begin{proof}
Since $\int_{\R^n}\chi_z(y)f(y)\,dy=1$, Jensen's inequality gives
\[
\widetilde u(\bar x_z)\leq\int_{\R^n}\widetilde u(y)\chi_z(y)f(y)\,dy=\int_{\R^n} u(y)\chi_z(y)f(y)\,dy\leq u(\bar x_z)+C,
\]
where we used Taylor's formula, $D^2u\leq\widetilde CI$, and $\operatorname{supp}\chi_z\subset B_C(\bar x_z)$. For any $x$, write $x=\sum_j\theta_jb_j$ as a convex combination of the vertices of a unit cube in the shifted lattice $\bar x_0+\mathbb Z^n$. Convexity and \eqref{eq:barycenterslattice} then give
\[
\widetilde u(x)\leq\sum_j\theta_j\widetilde u(b_j)\leq\sum_j\theta_j u(b_j)+C\leq u(x)+C,
\]
where the last inequality follows from Taylor's formula, $D^2u\leq\widetilde CI$, and $\sum_j\theta_j(b_j-x)=0$.

 For any $x$, choose $z$ with $|x-\bar x_z|\leq C$. Integrating the convexity inequality $2\widetilde u(y)\leq \widetilde u(x)+\widetilde u(2y-x)$ against $\chi_z(y)f(y)\,dy$, we obtain
\[
\widetilde u(x)\geq2\int_{\mathbb R^n}\chi_z(y) u(y) f(y)\,dy-\int\chi_z(y)u(2y-x)f(y)\,dy-C\geq u(x)-C,
\]
where the last inequality follows from Taylor's formula. The gradient bound then follows from the convexity of $\widetilde u$ and the $C^{1,1}$ regularity of $u$; see the proof of \cite[Lemma~A.3]{jin2025liouville}.
\end{proof}
 
To preserve the constraint under localization, we use a smooth right inverse $\Phi$ of $\Pi$ with uniformly bounded kernels.

\begin{Lemma}
\label{lem:cell-reconstruction}
Assume \eqref{eq:reference-cell-mass}. Let $a:\mathbb Z^n\to\mathbb R$. There exist smooth kernels $\phi_z(x)=\phi_0(x-z)$, supported in the interior of $Q_z$, such that
\begin{equation}
\label{eq:reconstruction-kernel-identities}
\Pi[\phi_{\widetilde z}](z)=\delta_{z\widetilde z}=
\begin{cases}
1,&\widetilde z=z\\
0,&\widetilde z\neq z
\end{cases}
,\qquad \Pi[\Phi[a]](z)=a(z),
\end{equation}
where
\begin{equation}
\label{eq:Phia}
\Phi[a](x):=\sum_{z\in\mathbb Z^n}a(z)\phi_z(x).
\end{equation}
Moreover, for each integer $j\geq0$,
\[
\|D^j\phi_z\|_{L^\infty(\mathbb R^n)}\leq C_j(n).
\]
\end{Lemma}

\begin{proof}
Define
\[
\phi_0(x):=\frac{\bigl(\mathbf1_{(1-\frac1{25n})Q_0}*\eta\bigr)(x)}{\displaystyle\int_{\mathbb R^n}\bigl(\mathbf1_{(1-\frac1{25n})Q_0}*\eta\bigr)f\,dx},\qquad \phi_z(x):=\phi_0(x-z).
\]
By \eqref{eq:reference-cell-mass}, the denominator is at least $3/4$. On $\operatorname{supp}\phi_z$, we have $\chi_z=1$ and $\chi_{\widetilde z}=0$ for $\widetilde z\ne z$. Consequently, \eqref{eq:reconstruction-kernel-identities} holds. The derivative bounds follow from the definition of $\phi_0$ and $\eta$.
\end{proof}

\subsection{Convex potentials with prescribed cell averages}
\label{sec5:constrained-potentials}

Let $L\geq2$ be an integer and set
\[
\mathbb T_L^n:=\mathbb R^n/L\mathbb Z^n,\qquad\Lambda_L:=(\mathbb Z/L\mathbb Z)^n.
\]
We identify functions on $\mathbb T_L^n$ and $\Lambda_L$ with their periodic lifts to $\mathbb R^n$ and $\mathbb Z^n$, respectively. We use the standard inner products on $L^2(\mathbb T_L^n)$ and $\ell^2(\Lambda_L)$:
\[
\langle g,\zeta\rangle_{L^2}:=\int_{\mathbb T_L^n}g\zeta\,dx,\qquad\langle a,b\rangle_{\ell^2}:=\sum_{z\in\Lambda_L}a(z)b(z).
\]
The operators of Section~\ref{sec5:cell-reconstruction} induce the maps
\[
\Pi:C^\infty(\mathbb T_L^n)\longrightarrow\mathbb R^{\Lambda_L},\qquad\Phi,\Pi^T:\mathbb R^{\Lambda_L}\longrightarrow C^\infty(\mathbb T_L^n).
\]
We write
$$
\ker\Pi:=\{\zeta\in H^1(\mathbb T_L^n):\Pi[\zeta]=0\}.
$$
By \eqref{eq:reconstruction-kernel-identities}, these maps satisfy $\Pi\Phi=I$. With respect to the above inner products, the transpose identity reads
\[
\langle\Pi[w],a\rangle_{\ell^2}=\langle w,\Pi^T[a]\rangle_{L^2}.
\]
Every smooth periodic function $g$ orthogonal to $\ker\Pi$ has the unique representation
\begin{equation}
\label{eq:cell-average-orthogonal-representation}
g=\Pi^T[a]=\left(\sum_{z\in\mathbb Z^n}a(z)\chi_z\right)f,\qquad a(z)=\int_{\mathbb R^n}g\phi_z\,dx.
\end{equation}
Indeed, $\Pi\Phi=I$ gives $\psi-\Phi[\Pi[\psi]]\in\ker\Pi$ for every smooth periodic test function $\psi$, and testing $g$ against this difference yields the representation.

Let $A$ be positive definite and let
$$
u=q_A+v,
$$
where $q_A(x)=\frac12x\cdot(Ax)$ and $v\in C^\infty(\mathbb T_L^n)$. Motivated by Lemma~\ref{lem:cell-average-convex-bound}, we seek a convex potential of the form
\[
\widetilde u=u+w,\qquad w\in C^\infty(\mathbb T_L^n)\cap\ker\Pi.
\]
Let $G$ be a given smooth positive function on $\mathbb T_L^n$ (which will be a small perturbation of the constant \(\det A\)). We seek $\widetilde u$ whose Monge--Amp\`ere density approximates $Gf$. Ideally, we would have $\det D^2\widetilde u=Gf$, but this equation is in general incompatible with the constraint $w\in\ker\Pi$. We therefore require the density error to vanish against every variation in $\ker\Pi$:
\[
\langle\det D^2\widetilde u-Gf,\zeta\rangle_{L^2}=0\qquad\forall\,\zeta\in\ker\Pi.
\]
By \eqref{eq:cell-average-orthogonal-representation}, we then conclude that
\begin{equation}
\label{eq:overview-density-error0}
\det D^2\widetilde u-G f=\left[\sum_z\bigl(\lambda(z)-\det D^2u(\bar x_z)\bigr)\chi_z\right]f,
\end{equation}
where
\begin{equation*}
\lambda(z)=\int_{\mathbb R^n}\phi_z\bigl(\det D^2\widetilde u-\omega f\bigr)\,dx \quad \mbox{with}\quad \omega(x)=G(x)-\sum_z\det D^2u(\bar x_z)\chi_z(x).
\end{equation*}
The equation \eqref{eq:overview-density-error0} can also be rewritten as
\[
\det D^2\widetilde u=\left(\sum_z\lambda(z)\chi_z+\omega\right)f.
\]

The following proposition shows that such a correction $w$ exists provided that $\|D^2v\|_\infty+\|\omega\|_\infty$ is sufficiently small.

\begin{Proposition}
\label{appsharp:blocks}
Let $A\in\mathcal K\Subset\mathcal S_+^{n\times n}$, and let $q_A(x):=\frac12x\cdot(Ax)$. Let $v,\omega\in C^\infty(\mathbb T_L^n)$. There exist constants $c,C>0$, depending only on $n$ and $\mathcal K$, such that if
\begin{equation}
\label{eq:constrained-data-smallness}
\mathcal N(v,\omega):=\|D^2v\|_\infty+\|\omega\|_\infty\leq c,
\end{equation}
then there exist a unique smooth convex potential $\widetilde u$ and a lattice function $\lambda:\Lambda_L\to\mathbb R$ with $\tfrac12\det A\leq\lambda(z)\leq2\det A$, such that $\widetilde u-q_A$ is $L\mathbb Z^n$-periodic and
\begin{equation}
\label{appweighted:additive-equations}
\det D^2\widetilde u=\left(\sum_z\lambda(z)\chi_z+\omega\right)f,\qquad \Pi[\widetilde u-q_A]=\Pi[v].
\end{equation}
Moreover, if $(\widehat u,\widehat\lambda)$ is the solution pair corresponding to another pair $(\widehat v,\widehat\omega)$ in the same parameter neighborhood, then
\begin{equation}
\label{appweighted:parameter-lipschitz}
\|(\widetilde u-v)-(\widehat u-\widehat v)\|_\infty+\|\lambda-\widehat\lambda\|_{\ell^\infty}\leq C\bigl(\|D^2(v-\widehat v)\|_\infty+\|\omega-\widehat\omega\|_\infty\bigr).
\end{equation}
\end{Proposition}

Proposition~\ref{appsharp:blocks} provides the constrained potential and its uniform stability. Its proof, based on uniform linearized estimates and the continuity method, is given in Section~\ref{sec5:linearized-estimates}. We next seek an estimate for
\[
|\lambda(z)-\det D^2u(\bar x_z)|
\]
that yields the desired density estimate through \eqref{eq:overview-density-error0}. To this end, we interpolate between the frozen and original data and analyze the first variation of the corresponding cell multiplier. The frozen corrector expansion in the next subsection shows that its leading $O(\varepsilon)$ contribution vanishes at the frozen cell. This step makes rigorous the cancellation suggested by the formal third-order expansion \eqref{eq:formal-third-order-expansion}.

\subsection{Cell multiplier estimates}
\label{appweighted:phase-source-section}

Throughout this subsection we assume that $u-q_A\in C^\infty(\mathbb T_L^n)$, where $L\geq2$ is an integer and $A\in\mathcal K\Subset\mathcal S_+^{n\times n}$. Let $\widehat G_\varepsilon$ be smooth, $L\mathbb Z^n$-periodic in $x$ and $\mathbb Z^n$-periodic in $y$, with uniform positive bounds, and set $G_\varepsilon(x)=\widehat G_\varepsilon(x,x)$. For $0<\varepsilon\leq1$, we assume that
\begin{equation}
\label{appweighted:torus-coefficients}
\begin{gathered}
\int_{\mathbb T^n}\widehat G_\varepsilon(x,y)f(y)\,dy=\det D^2u(x),\qquad\|D^3u\|_\infty\leq C\varepsilon,\qquad\|D_x\widehat G_\varepsilon\|_\infty\leq C\varepsilon,\\
[D^3u]_{C^{0,\alpha}}\leq C\varepsilon^{1+\alpha},\qquad\sup_{y\in\mathbb T^n}[D_x\widehat G_\varepsilon(\cdot,y)]_{C^{0,\alpha}}\leq C\varepsilon^{1+\alpha}.
\end{gathered}
\end{equation}
Unless otherwise specified, constants in the multiplier estimates depend only on $n$, $\alpha$, $\mathcal K$, and these bounds. They are independent of $\varepsilon$, $L$, the upper and lower bounds and derivatives of $f$, and higher regularity norms of the data beyond the stated bounds.

Let
\[
\omega(x)=G_\varepsilon(x)-\sum_z\det D^2u(\bar x_z)\chi_z(x).
\]

\begin{Proposition}
\label{appweighted:torus}
There exist constants $c_0,\varepsilon_0,C>0$ such that if \eqref{appweighted:torus-coefficients} holds, $\mathcal N(u-q_A,\omega)\leq c_0$ and $0<\varepsilon\leq\varepsilon_0$, then the pair $(\widetilde u,\lambda)$ given by Proposition~\ref{appsharp:blocks}, with $\widetilde u-q_A$ being $\mathbb T_L^n$-periodic, satisfies
\begin{equation}
\label{eq:direct-multiplier-consistency}
|\lambda(z)-\det D^2u(\bar x_z)|\leq C\varepsilon^{1+\alpha}\qquad\forall\,z\in\Lambda_L.
\end{equation}
Consequently,
\begin{equation}
\label{appweighted:corrected-density}
|\det D^2\widetilde u-G_\varepsilon f|\leq C\varepsilon^{1+\alpha}f,\qquad\|\widetilde u-u\|_\infty\leq C.
\end{equation}
Here, \(c_0\) depends only on \(n\) and \(\mathcal K\);  \(\varepsilon_0\) and \(C\) may also depend on the bounds in the coefficient assumptions.
\end{Proposition}

We prove Proposition~\ref{appweighted:torus} using Lemmas~\ref{lem:polynomial-linear-response} and~\ref{lem:localized-second-variation}, whose proofs are deferred to Section~\ref{sec5:localized-variations}. The first controls the multiplier error when comparing the first variation of the constrained potential with the frozen expansion, in terms of the Taylor remainders in the data. The second bounds the second variation of the cell multiplier by $C\varepsilon^2$. The constants are independent of $L$ and of the upper and lower bounds and derivatives of $f$.

\begin{proof}[Proof of Proposition~\ref{appweighted:torus}]
Fix $z_*\in\Lambda_L$. Translate the data and cell weights, relabel the cells, and subtract the corresponding linear terms from the potentials so that $z_*=\bar x_{z_*}=0$ and $\bar x_z=z$. For sufficiently small $c_0$, we have $D^2u>0$. Set
\[
A_0:=D^2u(0),\qquad G_0(y):=\widehat G_\varepsilon(0,y),\qquad\omega_0(y):=G_0(y)-\det A_0.
\]
Let $\widetilde u_0=q_{A_0}+v_0$ solve
\[
\det D^2\widetilde u_0=G_0f,\qquad v_0\in C^\infty(\mathbb T^n),\qquad\int_{\mathbb T^n}v_0f=0.
\]
By \eqref{appweighted:torus-coefficients}, the compatibility condition $\int_{\mathbb T^n}G_0f=\det A_0$ holds. The normalization gives $\Pi[\widetilde u_0-q_{A_0}]=0$, so the corresponding multiplier is $\lambda_0\equiv\det A_0$. Write
\[
(\widetilde U_0^{ij}):=\operatorname{cof}(D^2\widetilde u_0),\qquad\mathcal L_0[w]:=\widetilde U_0^{ij}w_{ij}.
\]
By the periodicity of $\omega_0$ and \eqref{appweighted:torus-coefficients},
$$
\|\omega_0\|_\infty\leq\|\omega\|_\infty+C\varepsilon.
$$
Set
\[
u_t=(1-t)q_{A_0}+tu,\quad\omega_t=(1-t)\omega_0+t\omega,\quad A_t=(1-t)A_0+tA,\quad 0\leq t\leq1.
\]
Enlarging $\mathcal K$ to a fixed compact subset of $\mathcal S_+^{n\times n}$ if necessary, we have $A_t\in\mathcal K$. For sufficiently small $c_0,\varepsilon_0$, Proposition~\ref{appsharp:blocks} and the implicit function argument in its proof give a smooth family $(\widetilde u_t,\lambda_t)$ satisfying
\[
\det D^2\widetilde u_t=\left(\sum_z\lambda_t(z)\chi_z+\omega_t\right)f,\qquad\Pi[\widetilde u_t-u_t]=0,\qquad\widetilde u_t-q_{A_t}\in C^\infty(\mathbb T_L^n).
\]
Uniqueness identifies the solution at $t=0$ with the frozen cell solution and gives $(\widetilde u_1,\lambda_1)=(\widetilde u,\lambda)$. Dots denote differentiation in $t$. Thus $\dot u=u-q_{A_0}$ and $\dot\omega=\omega-\omega_0$. Taylor's formula gives
\begin{equation}
\label{eq:direct-multiplier-taylor}
\lambda(0)-\det A_0=\dot\lambda_0(0)+\int_0^1(1-t)\ddot\lambda_t(0)\,dt.
\end{equation}
We estimate the two terms separately.

We now derive the equations for the first variation at $t=0$ and construct the corresponding frozen corrector expansion. Differentiating the determinant equation and the cell-average constraint for $(\widetilde u_t,\lambda_t)$ with respect to $t$ and evaluating at $t=0$, we obtain
\[
\mathcal L_0[\dot{\widetilde u}_0]=\left(\sum_z\dot\lambda_0(z)\chi_z+\dot\omega\right)f,\qquad\Pi[\dot{\widetilde u}_0]=\Pi[\dot u].
\]
For the first Taylor terms, set $B_i:=\partial_iD^2u(0)$ and $G_i(y):=\partial_{x_i}\widehat G_\varepsilon(0,y)$. The cubic term in the Taylor polynomial of $u$ is
\[
P(x):=\frac16D^3u(0)[x,x,x]=\frac16\sum_{i,j,k=1}^n\partial_{ijk}u(0)x_ix_jx_k,\qquad D^2P(x)=\sum_i x_iB_i.
\]
By \eqref{appweighted:torus-coefficients}, $|B_i|+\|G_i\|_\infty\leq C\varepsilon$. Differentiating the averaging identity in \eqref{appweighted:torus-coefficients}, we obtain
\[
\int_{\mathbb T^n}G_if=\operatorname{cof}A_0:B_i=:b_i.
\]
Let $\ell$ be the first-order Taylor polynomial of $u-q_{A_0}$ at $0$. Following \eqref{eq:formal-third-order-expansion}, we construct a model for $\dot{\widetilde u}_0$ in the form
\begin{equation*}
V=\ell+q_{A_0}+P+v_0+\sum_i x_iv_i+\psi_0-\widetilde u_0=\ell+P+\sum_i x_iv_i+\psi_0,
\end{equation*}
where $\ell+q_{A_0}+P$ is the third-order Taylor polynomial of $u$. The term $\sum_i x_iv_i$ corresponds to the variation of the second-order corrector, and $\psi_0$ corresponds to the third-order corrector. Thus $V$ is the correction to the frozen solution $\widetilde u_0$ suggested by the formal expansion \eqref{eq:formal-third-order-expansion}, after rescaling and Taylor expansion.

For $i=1,\ldots,n$, choose $v_i\in C^\infty(\mathbb T^n)$ satisfying
\[
\mathcal L_0[v_i]=G_if-\widetilde U_0^{jk}(B_i)_{jk},\qquad\int_{\mathbb T^n}v_if=0.
\]
Following \eqref{eq:formal-third-corrector}, choose $\psi_0\in C^\infty(\mathbb T^n)$ satisfying
\[
\mathcal L_0[\psi_0]=-2\sum_{i,j}\widetilde U_0^{ij}\partial_jv_i.
\]
Since
\[
\int_{\mathbb T^n}\widetilde U_0^{jk}\,dx=(\operatorname{cof}A_0)^{jk},\qquad\partial_j\widetilde U_0^{ij}=0,
\]
both right-hand sides have zero integral, and the periodic equations are solvable.

The zero $f$-means give $\Pi[v_0]=0$ and
\[
\begin{aligned}
\Pi\left[\sum_i x_iv_i+\psi_0\right](z)&=\sum_i z_i\int_{\mathbb T^n}v_if\,dx+\Pi\left[\sum_i x_iv_i+\psi_0\right](0)\\
&=\Pi\left[\sum_i x_iv_i+\psi_0\right](0).
\end{aligned}
\]
Subtracting this constant from $\psi_0$ and retaining the same notation, we obtain
\begin{equation}
\label{eq:frozen-expansion-average-compatibility}
\Pi[v_0]=0,\qquad\Pi\left[\sum_i x_iv_i+\psi_0\right]=0.
\end{equation}
Moreover,
\[
\begin{aligned}
\mathcal L_0[P]&=\sum_i x_i\,\widetilde U_0^{jk}(B_i)_{jk},\\
\mathcal L_0\left[\sum_i x_iv_i\right]&=\sum_i x_i\bigl(G_if-\widetilde U_0^{jk}(B_i)_{jk}\bigr)+2\sum_{i,j}\widetilde U_0^{ij}\partial_jv_i.
\end{aligned}
\]
The equation for $\psi_0$ cancels the last term. Thus, with
\[
\dot\omega_1(x):=\sum_i x_iG_i(x)-\sum_z(b\cdot z)\chi_z(x),
\]
we have
\begin{equation}
\label{eq:direct-frozen-first-response}
\mathcal L_0[V]=\left(\sum_i x_iG_i\right)f=\left(\sum_z(b\cdot z)\chi_z+\dot\omega_1\right)f,\qquad\Pi[V-\ell-P]=0.
\end{equation}
The model multiplier $b\cdot z$ vanishes at the frozen cell $z=0$. The construction uses only the equations for $v_i$ and $\psi_0$ and requires no uniform bounds for their Hessians. We use $V$ only in the linearized comparison, where convexity is not required.

Subtracting \eqref{eq:direct-frozen-first-response} from the system for $\dot{\widetilde u}_0$ gives
\[
\begin{aligned}
\mathcal L_0[\dot{\widetilde u}_0-V]&=\left(\sum_z\bigl(\dot\lambda_0(z)-b\cdot z\bigr)\chi_z+\dot\omega-\dot\omega_1\right)f,\\
\Pi[\dot{\widetilde u}_0-V]&=\Pi[\dot u-\ell-P].
\end{aligned}
\]
Since $\dot{\widetilde u}_0-\dot u$ and $v_1,\ldots,v_n,\psi_0$ are smooth and periodic, $(\dot{\widetilde u}_0-V)-(\dot u-\ell-P)$ and its gradient grow at most linearly.

By Taylor's formula and \eqref{appweighted:torus-coefficients},
\[
\begin{aligned}
|D^2(\dot u-\ell-P)(x)|&=\left|D^2u(x)-A_0-\sum_i x_iB_i\right|\leq C\varepsilon^{1+\alpha}|x|^{1+\alpha}.
\end{aligned}
\]
Similarly, using the averaging identity in \eqref{appweighted:torus-coefficients} and $|x-z|\leq C$ on $\operatorname{supp}\chi_z$, we have
\[
\begin{aligned}
&|\dot\omega(x)-\dot\omega_1(x)|\\
&=\Bigl|\widehat G_\varepsilon(x,x)-G_0(x)-\sum_i x_iG_i(x)-\sum_z\bigl(\det D^2u(z)-\det A_0-b\cdot z\bigr)\chi_z(x)\Bigr|\\
&\leq C\varepsilon^{1+\alpha}(1+|x|)^2.
\end{aligned}
\]
Applying \eqref{eq:polynomial-multiplier-response} in Lemma~\ref{lem:polynomial-linear-response} to $\dot{\widetilde u}_0-V$, with multiplier $\dot\lambda_0(z)-b\cdot z$, we obtain
\[
|\dot\lambda_0(z)-b\cdot z|\leq C\varepsilon^{1+\alpha}(1+|z|)^2.
\]
At $z=0$, this gives $|\dot\lambda_0(0)|\leq C\varepsilon^{1+\alpha}$.

The definitions of $\omega,\omega_0$ and \eqref{appweighted:torus-coefficients} give
\[
|D^2\dot u(x)|+|\dot\omega(x)|\leq C\varepsilon(1+|x|).
\]
Moreover, $\dot{\widetilde u}_t-\dot u$ is periodic. Lemma~\ref{lem:localized-second-variation}, with $x_*=0$ and $\delta=C\varepsilon$, therefore yields $|\ddot\lambda_t(0)|\leq C\varepsilon^2$ for $0\leq t\leq1$.

Hence, it follows from \eqref{eq:direct-multiplier-taylor} that
\[
|\lambda(0)-\det A_0|\leq C\varepsilon^{1+\alpha}+C\varepsilon^2\leq C\varepsilon^{1+\alpha}.
\]
Since $z_*$ was arbitrary, this proves \eqref{eq:direct-multiplier-consistency}.

Finally, since
\begin{equation*}
\det D^2\widetilde u-G_\varepsilon f=\left[\sum_z\bigl(\lambda(z)-\det D^2u(\bar x_z)\bigr)\chi_z\right]f,
\end{equation*}
the density estimate in \eqref{appweighted:corrected-density} follows from \eqref{eq:direct-multiplier-consistency}, $\chi_z\geq0$, and $\sum_z\chi_z=1$. The bound for $\widetilde u-u$ follows from Lemma~\ref{lem:cell-average-convex-bound}.
\end{proof}

\subsection{Proof of the higher-order estimate}
\label{appweighted:localization}

Let $u$ be the solution of \eqref{eq:weighted-effective-problem}. We complete the proof of Theorem~\ref{thm:weighted-measure-holder-rate} by verifying \eqref{eq:local-comparison-hypotheses} with $\beta=1+\alpha$ and $V^\varepsilon=u$.

\begin{proof}[Proof of the case $m=1$ of Theorem~\ref{thm:weighted-measure-holder-rate}]
By the Alexandrov maximum principle, it suffices to consider sufficiently small $\varepsilon$. After a fixed translation and dilation, we may assume that $\overline\Omega\subset B_1$. By the approximation argument in Section~\ref{sec:measure-approximation}, it suffices to work with smooth data on the fixed neighborhood $\mathcal O$ of $\Omega$, retaining the notation $u,F,\mu$ for the approximations.

We first make a localization. Fix $x_*\in\overline\Omega$, and let $q_0$ be the quadratic Taylor polynomial of $u$ at $x_*$. We shall apply Proposition~\ref{appweighted:torus} with
\[
\mathcal K:=\{D^2u(x):x\in\overline\Omega\}\Subset\mathcal S_+^{n\times n}.
\]
Choose $r_1>0$, depending only on the original data, small enough that $B_{r_1/4}(x_*)\subset\mathcal O$. Let $0\leq\eta_0\leq1$ be a smooth cutoff, equal to one on $B_{r_1/5}(x_*)$ and zero outside $B_{r_1/4}(x_*)$, and set
\[
\widehat u_0=q_0+\eta_0(u-q_0).
\]
Taylor's formula gives
\[
\|D^2\widehat u_0-D^2u(x_*)\|_\infty\leq Cr_1,\qquad\|D^3\widehat u_0\|_{C^{0,\alpha}(\mathbb R^n)}\leq C(1+r_1^{-\alpha}).
\]
To localize the coefficient, we write $g=F/\overline F$ and set
\begin{align*}
g_0(x,y)&=g(x_*,y)+\eta_0(x)(g(x,y)-g(x_*,y)),\\
d\mu_0(y)&=g(x_*,y)\,d\mu(y)=:f(y)\,dy.
\end{align*}
The normalization of $g$ gives
\[
\int_{\mathbb T^n}f(y)\,dy=\int_{\mathbb T^n}g_0(x,y)\,d\mu(y)=1.
\]
Moreover,
\[
0<c\leq g_0\leq C,\qquad\sup_y\|g_0(\cdot,y)\|_{C^{1,\alpha}}\leq C(1+r_1^{-\alpha}),\qquad\|g_0-g(x_*,\cdot)\|_\infty\leq Cr_1.
\]

We next make a rescaling and construct the constrained potential. Set $A:=D^2u(x_*)$. In the coordinates $\xi=x/\varepsilon$, define
\begingroup
\[
\widehat u(\xi)=q_A(\xi)+\varepsilon^{-2}\bigl[\widehat u_0(\varepsilon\xi)-q_0(\varepsilon\xi)\bigr],\qquad\widehat G_\varepsilon(\xi,y)=\det D^2\widehat u_0(\varepsilon\xi)\frac{g_0(\varepsilon\xi,y)}{g(x_*,y)}.
\]
\endgroup
Since $D^2q_0=A$, we have
\[
D_\xi^2\widehat u(\xi)=D^2\widehat u_0(\varepsilon\xi),\qquad D_\xi^3\widehat u(\xi)=\varepsilon D^3\widehat u_0(\varepsilon\xi).
\]
Choose an integer $L\geq2$ with $r_1\leq L\varepsilon\leq2r_1$ and write $\mathbb T_L^n:=\mathbb R^n/L\mathbb Z^n$. Periodize $\widehat u-q_A$ and $\widehat G_\varepsilon-\det A$ in $\xi$ with period $L$. Distinct translates of their supports in $\xi$ are disjoint, so periodization leaves the data unchanged on $B_{r_1/4}(x_*)$ in the original variables. Set $G_\varepsilon(\xi):=\widehat G_\varepsilon(\xi,\xi)$. The normalization of $g_0$ and the scaling give \eqref{appweighted:torus-coefficients} for $(\widehat u,\widehat G_\varepsilon)$:
\[
\begin{gathered}
\int_{\mathbb T^n}\widehat G_\varepsilon(\xi,y)f(y)\,dy=\det D^2\widehat u(\xi),\qquad\|D^3\widehat u\|_\infty\leq C\varepsilon,\qquad\|D_\xi\widehat G_\varepsilon\|_\infty\leq C\varepsilon,\\
[D^3\widehat u]_{C^{0,\alpha}(\mathbb R^n)}\leq C\varepsilon^{1+\alpha},\qquad\sup_{y\in\mathbb T^n}[D_\xi\widehat G_\varepsilon(\cdot,y)]_{C^{0,\alpha}(\mathbb R^n)}\leq C\varepsilon^{1+\alpha}.
\end{gathered}
\]
For the cell weights from Section~\ref{sec5:cell-reconstruction}, set
\[
\omega(\xi)=G_\varepsilon(\xi)-\sum_z\det D^2\widehat u(\bar x_z)\chi_z(\xi).
\]
Since $|\xi-\bar x_z|\leq C$ for all $\xi\in\operatorname{supp}\chi_z$, the preceding bounds give
\[
\|D^2\widehat u-A\|_\infty+\|\omega\|_\infty\leq C(r_1+\varepsilon).
\]
This is the smallness condition \eqref{eq:constrained-data-smallness} required in Proposition~\ref{appweighted:torus}. After translating $\mu_0$ by a suitable $y_0$ and replacing $\widehat G_\varepsilon(\xi,y)$ and $\widehat u(\xi)$ by $\widehat G_\varepsilon(\xi+y_0,y+y_0)$ and $\widehat u(\xi+y_0)-[q_A(\xi+y_0)-q_A(\xi)]$, respectively, we may assume that \eqref{eq:reference-cell-mass} holds.

Choosing $r_1$ and then $\varepsilon$ sufficiently small, Proposition~\ref{appweighted:torus} gives a convex potential $\widetilde u$, with $\widetilde u-q_A$ being $\mathbb T_L^n$-periodic, such that
\[
\Pi[\widetilde u-\widehat u]=0,\qquad\|\widetilde u-\widehat u\|_\infty\leq C,\qquad|\det D^2\widetilde u-G_\varepsilon f|\leq C\varepsilon^{1+\alpha}f.
\]
These are the two approximation estimates needed below. Undoing the translation by $y_0$, we replace $\widetilde u$ by
\[
\xi\longmapsto \widetilde u(\xi-y_0)+q_A(\xi)-q_A(\xi-y_0)
\]
and retain the notation $\widetilde u$. We also restore the original notation for $\widehat u$, $G_\varepsilon$, and $f$.

We finally make a local comparison. Define
\[
W_{x_*}(x):=q_0(x)+\varepsilon^2\left[\widetilde u(x/\varepsilon)-q_A(x/\varepsilon)\right].
\]
Since the cutoff satisfies $\eta_0\equiv1$ in $B_{r_1/5}(x_*)$, we have
\[
\widehat u_0=u,\qquad D^2\widehat u_0=D^2u,\qquad g_0=g\quad\text{on }B_{r_1/5}(x_*).
\]
By the definitions of $\widehat u$ and $W_{x_*}$,
\[
D^2W_{x_*}(x)=D^2\widetilde u(x/\varepsilon),\quad\mbox{and }W_{x_*}(x)-u(x)=\varepsilon^2(\widetilde u-\widehat u)(x/\varepsilon)\quad\mbox{on }B_{r_1/5}(x_*).
\]
Thus $W_{x_*}$ is convex and $\|W_{x_*}-u\|_{L^\infty(B_{r_1/5}(x_*))}\leq C\varepsilon^2$. On the same ball, $x=\varepsilon\xi$ and $\det D^2\widehat u_0(\varepsilon\xi) =\det D^2u(x)=\overline F(x)$. Thus
\[
\begin{aligned}
G_\varepsilon(x/\varepsilon)=\det D^2u(x)\,\frac{g(x,x/\varepsilon)}{g(x_*,x/\varepsilon)}=\frac{F(x,x/\varepsilon)}{g(x_*,x/\varepsilon)}.
\end{aligned}
\]
Since $f(y)\,dy=g(x_*,y)\,d\mu(y)$, rescaling gives
\[
\begin{aligned}
f(x/\varepsilon)\,dx&=g(x_*,x/\varepsilon)\,d\mu_\varepsilon(x),\\
G_\varepsilon(x/\varepsilon)f(x/\varepsilon)\,dx&=F(x,x/\varepsilon)\,d\mu_\varepsilon(x).
\end{aligned}
\]
Set $\ell(x)=u(x_*)+Du(x_*)\cdot(x-x_*)$. By \eqref{appweighted:corrected-density}, Taylor's formula, and $F\geq\lambda>0$, we obtain the following estimates. With $r_1$ fixed, choose $C_0$ sufficiently large and $\varepsilon_0>0$ sufficiently small. Then, for every $0<\varepsilon\leq\varepsilon_0$,
\begin{equation*}
\begin{gathered}
|W_{x_*}-\ell|\leq C(r_1^2+\varepsilon^2)\leq C_0r_1^2,\qquad|W_{x_*}-u|\leq C\varepsilon^2\leq C_0\varepsilon^{1+\alpha}r_1^2,\\
(1-C_0\varepsilon^{1+\alpha})\det D^2u^\varepsilon\leq\det D^2W_{x_*}\leq(1+C_0\varepsilon^{1+\alpha})\det D^2u^\varepsilon
\end{gathered}
\end{equation*}
in $\Omega\cap B_{r_1/6}(x_*)$. Here we used $\varepsilon^2\leq\varepsilon^{1+\alpha}$ and absorbed the fixed factor $r_1^{-2}$ into $C_0$. Applying Proposition~\ref{prop:local-comparison} with $r=r_1/6$, $\beta=1+\alpha$, and $V^\varepsilon=u$, we obtain
\[
\|u^\varepsilon-u\|_{L^\infty(\Omega)}\leq C\varepsilon^{1+\alpha}.
\]
Passing to the limit in the approximation gives the same estimate for the original data.
\end{proof}

\subsection{Uniform estimates for the constrained problem}
\label{sec5:linearized-estimates}

In this section, we prove Proposition~\ref{appsharp:blocks} using the averaging and reconstruction operators from Section~\ref{sec5:cell-reconstruction} and Lemmas~\ref{lem:cell-average-convex-bound} and~\ref{lem:cell-reconstruction}. The argument has three parts: solve \eqref{eq:linearized-constrained-problem}, establish the uniform bound \eqref{eq:constrained-linear-response-bound}, and apply the continuation method from $\det D^2\widetilde u_0=(\det A)f$ to the prescribed data.

Write $u=q_A+v$. Throughout this section, $B_R(x_0)$ denotes a Euclidean ball. Fix $0<\widetilde c\leq\widetilde C$ and assume
\begin{equation}
\label{eq:linearized-background-assumptions}
\begin{gathered}
\widetilde u-q_A,\ u-q_A\in C^\infty(\mathbb T_L^n),\qquad \widetilde cI\leq D^2u\leq\widetilde CI,\\
\widetilde u\text{ is convex},\qquad \widetilde cf\leq\rho:=\det D^2\widetilde u\leq\widetilde Cf,\qquad \Pi[\widetilde u-u]=0.
\end{gathered}
\end{equation}
Let $\dot u,\dot\omega\in C^\infty(\mathbb T_L^n)$ denote prescribed first variations of the background and additive density, respectively. Linearizing \eqref{appweighted:additive-equations} at $\widetilde u$ gives the following system for the corresponding variations $\dot{\widetilde u}$ and $\dot\lambda$:
\begin{equation}
\label{eq:linearized-constrained-problem}
\mathcal L[\dot{\widetilde u}]=\Pi^T[\dot\lambda]+\dot\omega f=\left(\sum_z\dot\lambda(z)\chi_z+\dot\omega\right)f,\qquad \Pi[\dot{\widetilde u}-\dot u]=0,
\end{equation}
where $\mathcal L$ is the linearized Monge--Amp\`ere operator at $\widetilde u$,
\[
\mathcal L[w]:=\widetilde U^{ij}w_{ij}=\partial_i(\widetilde U^{ij}w_j)=\partial_i\partial_j(\widetilde U^{ij}w),
\]
with $(\widetilde u^{ij}):=(D^2\widetilde u)^{-1}$ and $(\widetilde U^{ij}):=\operatorname{cof}(D^2\widetilde u) =\rho(\widetilde u^{ij})$. To solve this system, we use the symmetric bilinear form associated with $-\mathcal L$ and its quadratic form:
\[
\mathcal E(w,\zeta):=\int_{\mathbb T_L^n}\langle Dw,D\zeta\rangle_{\widetilde u}\,dx=-\langle\mathcal L[w],\zeta\rangle_{L^2},\qquad \mathcal E(w):=\mathcal E(w,w),
\]
where
\[
\langle p,q\rangle_{\widetilde u}:=\widetilde U^{ij}p_iq_j,\qquad |p|_{\widetilde u}^2:=\langle p,p\rangle_{\widetilde u}\qquad\forall\,p,q\in\mathbb R^n.
\]
The form $\mathcal E$ and the averaging operator $\Pi$ extend continuously to $H^1(\mathbb T_L^n)$; we retain the same notation for these extensions. Unless another integration domain is specified, $\mathcal E$ denotes the energy over $\mathbb T_L^n$. Testing \eqref{eq:linearized-constrained-problem} against $\phi_z$ (which is the one in Lemma \ref{lem:cell-reconstruction}) gives
\begin{equation}
\label{eq:constrained-multiplier-recovery}
\dot\lambda(z)=-\mathcal E(\dot{\widetilde u},\phi_z)-\langle\dot\omega f,\phi_z\rangle_{L^2}.
\end{equation}

\begin{Lemma}
\label{lem:linearized-constrained-solvability}
Assume \eqref{eq:linearized-background-assumptions}. For every $\dot u,\dot\omega\in C^\infty(\mathbb T_L^n)$, the equation \eqref{eq:linearized-constrained-problem} admits a unique smooth solution $(\dot{\widetilde u},\dot\lambda)$. The difference $\dot\zeta:=\dot{\widetilde u}-\dot u$ is the unique element of $\ker\Pi$ satisfying
\begin{equation}
\label{eq:constrained-direct-variation}
\mathcal E(\dot\zeta,\xi)=-\mathcal E(\dot u,\xi)-\langle\dot\omega f,\xi\rangle_{L^2}\qquad\forall\,\xi\in\ker\Pi.
\end{equation}
\end{Lemma}

\begin{proof}
For any solution, the average constraint gives $\dot\zeta=\dot{\widetilde u}-\dot u\in\ker\Pi$. For every $\xi\in\ker\Pi$, the linearized equation and the transpose identity give
\[
\mathcal E(\dot{\widetilde u},\xi)=-\langle\mathcal L[\dot{\widetilde u}],\xi\rangle_{L^2}=-\langle\Pi^T[\dot\lambda]+\dot\omega f,\xi\rangle_{L^2}=-\langle\dot\lambda,\Pi[\xi]\rangle_{\ell^2}-\langle\dot\omega f,\xi\rangle_{L^2}=-\langle\dot\omega f,\xi\rangle_{L^2}.
\]
This yields \eqref{eq:constrained-direct-variation}. Conversely, coercivity on $\ker\Pi$ for the fixed smooth data gives a unique solution $\dot\zeta$ of \eqref{eq:constrained-direct-variation}, and hence determines $\dot{\widetilde u}$. The residual $\mathcal L[\dot{\widetilde u}]-\dot\omega f$ annihilates $\ker\Pi$. Since $\Pi\Phi=I$, it has a unique representation $\Pi^T[\dot\lambda]$. Elliptic regularity then yields smoothness.
\end{proof}

We first control the local mass of $\rho+\operatorname{tr}(\widetilde U^{ij})$ and the energy of $\Phi[a]$ defined in \eqref{eq:Phia}. These bounds do not require pointwise bounds for $D^2\widetilde u$.

\begin{Lemma}
\label{lem:local-coefficient-reconstruction}
Assume \eqref{eq:linearized-background-assumptions}. There exists a constant $C>0$, depending only on $n,\widetilde c,\widetilde C$, such that
\begin{equation}
\label{eq:constrained-local-coefficient-bound}
\int_{B_R(x_0)}\det(D^2\widetilde u+I)\,dx\leq CR^n\qquad\forall\,x_0\in\mathbb R^n,\ R\geq1.
\end{equation}
Moreover, for every $a:\Lambda_L\to\mathbb R$,
\begin{equation}
\label{eq:constrained-reconstruction-bound}
\mathcal E(\Phi[a])+\|\Phi[a]\|_{L^2(\mathbb T_L^n,\rho\,dx)}^2\leq C\|a\|_{\ell^2(\Lambda_L)}^2.
\end{equation}
\end{Lemma}

\begin{proof}
The gradient bound in Lemma~\ref{lem:cell-average-convex-bound} and the Hessian bound for $u$ give
\[
\int_{B_R(x_0)}\det(D^2\widetilde u+I)\,dx=\bigl|(D\widetilde u+x)(B_R(x_0))\bigr|\leq CR^n.
\]
Since
\begin{equation}
\label{eq:aux00}
\left(\rho+\operatorname{tr}(\widetilde U^{ij})\right)\leq C_n\det(D^2\widetilde u+I),
\end{equation}
the support and derivative bounds for the kernels $\phi_z$ in Lemma~\ref{lem:cell-reconstruction}, together with \eqref{eq:constrained-local-coefficient-bound}, then yield
\[
\mathcal E(\Phi[a])+\|\Phi[a]\|_{L^2(\mathbb T_L^n,\rho\,dx)}^2\leq C\sum_{z\in\Lambda_L}|a(z)|^2\int_{B_{4n}(z)}\left(\rho+\operatorname{tr}(\widetilde U^{ij})\right)\,dx\leq C\|a\|_{\ell^2(\Lambda_L)}^2.
\]
\end{proof}

The next two estimates give coercivity on $\ker\Pi$ and control the terms containing derivatives of a cutoff.

\begin{Lemma}
Assume \eqref{eq:linearized-background-assumptions}. Then there exists a constant $C>0$, depending only on $n,\widetilde c,\widetilde C$, such that the following estimates hold for every $v\in C^\infty(\mathbb T_L^n)$:
\begin{equation}
\label{appsharp:augmented-gap}
\|v\|_{L^2(\mathbb T_L^n,\rho\,dx)}^2\leq C\bigl(\mathcal E(v)+\|\Pi[v]\|_{\ell^2(\Lambda_L)}^2\bigr).
\end{equation}

\begin{equation}
\label{appsharp:global-weighted-trace}
\int_{\mathbb T_L^n}\operatorname{tr}(\widetilde U^{ij})v^2\,dx\leq C\bigl(\mathcal E(v)+\|v\|_{L^2(\mathbb T_L^n,\rho\,dx)}^2\bigr).
\end{equation}
\end{Lemma}

\begin{proof}
\begingroup
We recall a variant of the classical Brascamp--Lieb inequality (see \cite[Theorem~9, with $\beta=0$]{nguyen2014variance}): for a smooth positive concave function $\phi$ on a ball $B$, with $-D^2\phi>0$, and any smooth function $V$ on $\overline B$,
\[
\int_B|V-V_B|^2\,dy\leq\int_B\phi\,(-D^2\phi)^{-1}DV\cdot DV\,dy,
\]
where $V_B:=|B|^{-1}\int_BV\,dy$ is the ordinary average. For a smooth convex function $\psi$ with $D^2\psi>0$, fix $y_0\in B$ and set $M:=\operatorname{diam}(B)\operatorname{diam}(D\psi(B))$ and $\phi(y):=2M-\psi(y)+\psi(y_0)+D\psi(y_0)\cdot(y-y_0)$. Since $M\leq\phi\leq2M$ and $-D^2\phi=D^2\psi$, we obtain
\begin{equation}
\label{eq:convex-potential-poincare}
\int_B|V-V_B|^2\leq C\operatorname{diam}(B)\operatorname{diam}(D\psi(B))\int_B(D^2\psi)^{-1}DV\cdot DV.
\end{equation}
\par\endgroup

We now apply \eqref{eq:convex-potential-poincare} to the Legendre dual $\psi=(\widetilde u)^*$. Extend $v$ to $\mathbb R^n$ by $L\mathbb Z^n$-periodicity and set $V(y):=v(D\psi(y))$. By the uniform convexity of $u$ and Lemma~\ref{lem:cell-average-convex-bound},
\[
\|D^2(u^*)\|_\infty\leq C,\qquad \|D\psi-D(u^*)\|_\infty\leq C.
\]
In dual coordinates, the prescribed averages take the form
\[
\Pi[v](z)=\int_{\mathbb R^n}V\theta_z\,dy,\qquad\theta_z(y)=\frac{f(D\psi(y))}{\rho(D\psi(y))}\chi_z(D\psi(y)),\qquad\forall\,z\in\mathbb Z^n.
\]
Since $f$ is normalized and $\mathbb Z^n$-periodic, $\int_{\mathbb R^n}\chi_z f\,dx =\int_{\mathbb R^n}\eta(y)\left(\int_{Q_z+y}f(x)\,dx\right)dy=1$, so the weights $\theta_z$ have unit mass. Moreover, $0\leq\theta_z\leq C$ and
\[
\operatorname{supp}\theta_z\subset D\widetilde u(\operatorname{supp}\chi_z)\subset B_C\!\bigl(Du(z)\bigr)=:B_z,
\]
with $C$ independent of $L$. The gradient bounds above imply $\operatorname{diam}(D\psi(B_z))\leq C$. Consequently,
\[
|\Pi[v](z)-V_{B_z}|^2=\left|\int_{B_z}(V-V_{B_z})\theta_z\,dy\right|^2\leq C\int_{B_z}|V-V_{B_z}|^2\,dy.
\]
Combining this with \eqref{eq:convex-potential-poincare} on $B_z$ yields
\[
\int_{B_z}V^2\leq C\left(\int_{B_z}(D^2\psi)^{-1}DV\cdot DV+|\Pi[v](z)|^2\right).
\]

Since $D\widetilde u(x+Lz)=D\widetilde u(x)+LAz$ for $z\in\mathbb Z^n$, the functions $V$ and $D^2\psi$ are $LA\mathbb Z^n$-periodic. Let $P:=D\widetilde u([0,L)^n)$ be a fundamental domain for $LA\mathbb Z^n$. The change of variables $y=D\widetilde u(x)$ gives
\begin{equation}
\label{sec5:dual-energy}
\int_{\mathbb T_L^n}|v|^p\rho\,dx=\int_P|V|^p\,dy\quad\forall\,p>0,\qquad\mathcal E(v)=\int_P (D^2\psi)^{-1}DV\cdot DV\,dy.
\end{equation}
After periodization, the balls and their fixed enlargements cover the dual torus with uniformly bounded overlap. Summing over $z\in\Lambda_L$ and using \eqref{sec5:dual-energy}, we obtain \eqref{appsharp:augmented-gap}.

Finally, since $\partial_i\widetilde U^{ij}=0$ and $\widetilde U^{ij}\widetilde u_{ij}=n\rho$, integration by parts yields
\[
\begin{aligned}
\int_{\mathbb T_L^n}\operatorname{tr}(\widetilde U^{ij})v^2\,dx&\leq C\int_{\mathbb T_L^n}v^2\widetilde U^{ij}u_{ij}\,dx\\
&=C\int_{\mathbb T_L^n}v^2\left(n\rho-\partial_i\bigl(\widetilde U^{ij}(\widetilde u-u)_j\bigr)\right)\,dx\\
&=C\left(n\int_{\mathbb T_L^n}v^2\rho\,dx+2\int_{\mathbb T_L^n}v\langle D\widetilde u-Du,Dv\rangle_{\widetilde u}\,dx\right)\\
&\leq C\int_{\mathbb T_L^n}v^2\rho\,dx+C\left(\int_{\mathbb T_L^n}v^2|D\widetilde u-Du|_{\widetilde u}^2\,dx\right)^{1/2}\mathcal E(v)^{1/2}\\
&\leq C\int_{\mathbb T_L^n}v^2\rho\,dx+C\left(\int_{\mathbb T_L^n}\operatorname{tr}(\widetilde U^{ij})v^2\,dx\right)^{1/2}\mathcal E(v)^{1/2}.
\end{aligned}
\]
Here the last inequality uses $\|D\widetilde u-Du\|_\infty\leq C$ from Lemma~\ref{lem:cell-average-convex-bound}. Young's inequality yields \eqref{appsharp:global-weighted-trace}.
\end{proof}

The following local estimate converts localized $L^2$ bounds into pointwise bounds.

\begin{Lemma}
\label{lem:linearized-local-pointwise}
Assume \eqref{eq:linearized-background-assumptions}. Then there exist constants $R,C>0$, depending only on $n,\widetilde c,\widetilde C$, such that every $v\in C^\infty(\mathbb R^n)$ satisfies
\begin{equation*}
|v(x)|\leq\frac14\sup_{B_R(x)}|v|+C\|v\|_{L^2(B_R(x),\rho\,dy)}+C\|\rho^{-1}\mathcal L[v]\|_{L^\infty(B_R(x))}\qquad\forall\,x\in\mathbb R^n.
\end{equation*}
\end{Lemma}

\begin{proof}
By translation and subtraction of a supporting affine function, we may assume that $x=0$, $\widetilde u\geq0$, and $\widetilde u(0)=0$. Set $w=(\widetilde u+n)/(8n)$. By Lemma~\ref{lem:cell-average-convex-bound}, $\widetilde u(y)\geq c|y|^2-C$. Since $\rho^{-1}\mathcal L[\widetilde u]=n$, we can choose a fixed $R$ such that $w\geq1$ on $\partial B_R$, while $w(0)=1/8$ and $(\rho^{-1}\mathcal L-I)[w]\leq0$.

Put $M=\sup_{B_R}|v|$. The functions $\pm v-Mw$ are nonpositive on $\partial B_R$ and satisfy
\[
(\rho^{-1}\mathcal L-I)[\pm v-Mw]\geq-|v|-|\rho^{-1}\mathcal L[v]|.
\]
Applying the ABP maximum estimate \cite[Theorem~9.1]{gilbarg2001elliptic} to $\pm v-Mw$, with coefficient matrix $(\widetilde u^{ij})$, and using $w(0)=1/8$, we obtain
\[
|v(0)|\leq\frac{1}{8}M+C\|v\|_{L^n(B_R,\rho\,dy)}+C\|\rho^{-1}\mathcal L[v]\|_{L^n(B_R,\rho\,dy)}.
\]
Note that $\int_{B_R}\rho\,dy\leq C\int_{B_R}f\,dy\leq C$. The desired estimate then follows from Cauchy--Schwarz in dimension one and from interpolation and Young's inequality in higher dimensions, together with $\|\rho^{-1}\mathcal L[v]\|_{L^n(B_R,\rho\,dy)} \leq C\|\rho^{-1}\mathcal L[v]\|_{L^\infty(B_R)}$.

\end{proof}

For a domain $D$, write $\mathcal E(v;D):=\int_D|Dv|_{\widetilde u}^2\,dx$ and
\begin{equation}
\label{eq:constrained-local-energy}
J(v;D):=\mathcal E(v;D)+\int_D\left(\rho+\operatorname{tr}(\widetilde U^{ij})\right)v^2\,dx.
\end{equation}

\begin{Lemma}
\label{lem:constrained-local-response-energy}
Under the assumptions of Lemma~\ref{lem:linearized-constrained-solvability}, let $(\dot{\widetilde u},\dot\lambda)$ solve \eqref{eq:linearized-constrained-problem} and set $\dot\zeta:=\dot{\widetilde u}-\dot u$. For every $x_0\in\mathbb R^n$ and $R\geq1$, we have
\begin{equation}
\label{eq:constrained-local-response-energy}
J(\dot\zeta;B_R(x_0))\leq C_R\bigl(\|D^2\dot u\|_\infty+\|\dot\omega\|_\infty\bigr)^2,
\end{equation}
where $C_R>0$ depends only on $n,\widetilde c,\widetilde C$ and $R$.
\end{Lemma}

\begin{proof}
For each fixed radius $R$, replacing $L$ by an integer multiple, we may assume that $L\geq CR$. The compactly supported functions below are then identified with their periodizations on $\mathbb T_L^n$. Set $d:=\|D^2\dot u\|_\infty+\|\dot\omega\|_\infty$. Choose a fixed radius $R_0\geq1$ such that $B_{R_0}(\bar x_z)$ contains the supports of $\chi_z$ and $\phi_z$ for every $z$.

\medskip
\noindent\emph{A cutoff preserving the averages.} We first derive a cutoff estimate on $\ker\Pi$, also used in Section~\ref{sec5:localized-variations}. Let $\zeta\in\ker\Pi$, $R\geq1$, and choose a smooth cutoff $0\leq\eta\leq1$ with $|D\eta|\leq C/R$. Let $\mathcal O_\eta$ be the $2R_0$-neighborhood of $\operatorname{supp}D\eta$, and set
\[
\xi:=\eta^2\zeta-\Phi[\Pi[\eta^2\zeta]]\in\ker\Pi,\qquad Y:=\int_{\mathbb T_L^n}\eta^2|D\zeta|_{\widetilde u}^2\,dx.
\]
Since $\Pi[\zeta]=0$, the coefficient $\Pi[\eta^2\zeta](z)$ vanishes whenever $\eta$ is constant on $\operatorname{supp}\chi_z$. Hence $\operatorname{supp}\Phi[\Pi[\eta^2\zeta]]\subset\mathcal O_\eta$. Subtracting $\eta(\bar x_z)$ and $\eta^2(\bar x_z)$ in the respective averages and using \eqref{eq:constrained-reconstruction-bound}, we obtain
\[
\|\Pi[\eta\zeta]\|_{\ell^2}^2+\|\Pi[\eta^2\zeta]\|_{\ell^2}^2 +\mathcal E(\Phi[\Pi[\eta^2\zeta]])+\|\Phi[\Pi[\eta^2\zeta]]\|_{L^2(\rho\,dx)}^2\leq CR^{-2}J(\zeta;\mathcal O_\eta).
\]
The estimates \eqref{appsharp:augmented-gap}--\eqref{appsharp:global-weighted-trace}, applied to $\eta\zeta$ and $\xi$, therefore give
\begin{equation}
\label{eq:constrained-cutoff-energy}
\begin{aligned}
&\mathcal E(\xi)+\|\xi\|_{L^2(\rho\,dx)}^2+\int\eta^2\left(\rho+\operatorname{tr}(\widetilde U^{ij})\right)\zeta^2\,dx\leq C\bigl(Y+R^{-2}J(\zeta;\mathcal O_\eta)\bigr).
\end{aligned}
\end{equation}
Expanding $\mathcal E(\zeta,\eta^2\zeta)$ and using $\operatorname{supp}\Phi[\Pi[\eta^2\zeta]]\subset\mathcal O_\eta$, we obtain by Cauchy--Schwarz
\[
\begin{aligned}
Y&=\mathcal E(\zeta,\xi)-2\int_{\mathbb T_L^n}\eta\zeta\,\langle D\zeta,D\eta\rangle_{\widetilde u}\,dx+\mathcal E(\zeta,\Phi[\Pi[\eta^2\zeta]])\\
&\leq|\mathcal E(\zeta,\xi)|+\tfrac12Y+CR^{-1}J(\zeta;\mathcal O_\eta).
\end{aligned}
\]
Absorbing $\tfrac12Y$ into the left-hand side and using \eqref{eq:constrained-cutoff-energy}, we obtain
\begin{equation}
\label{eq:constrained-cutoff-test}
\int_{\mathbb T_L^n}\eta^2\left[|D\zeta|_{\widetilde u}^2+\left(\rho+\operatorname{tr}(\widetilde U^{ij})\right)\zeta^2\right]\,dx\leq C|\mathcal E(\zeta,\xi)|+CR^{-1}J(\zeta;\mathcal O_\eta).
\end{equation}

\medskip
\noindent\emph{Testing the equation.} For $R\geq2R_0$, we apply the cutoff estimate \eqref{eq:constrained-cutoff-test} to $\dot\zeta$ with a cutoff $\eta$ equal to one on $B_R(x_0)$ and supported in $B_{2R}(x_0)$. Then $\operatorname{supp}\xi\subset B_{3R}(x_0)$. Testing \eqref{eq:constrained-direct-variation} with $\xi$ yields
\[
\mathcal E(\dot\zeta,\xi)=-\mathcal E(\dot u,\xi)-\langle\dot\omega f,\xi\rangle_{L^2}.
\]
Using $\partial_i\widetilde U^{ij}=0$, we may subtract $D\dot u(x_0)$ in the background term so that
\[
\mathcal E(\dot u,\xi)=\int_{B_{3R}(x_0)} \left\langle D\dot u-D\dot u(x_0),D\xi\right\rangle_{\widetilde u}.
\]
Since
\[
|D\dot u-D\dot u(x_0)|\leq3R\sup_{B_{3R}(x_0)}|D^2\dot u|\qquad\forall\,x\in\operatorname{supp}\xi,
\]
we have
\[
\begin{aligned}
|\mathcal E(\dot u,\xi)|&\le\left(\int_{B_{3R}(x_0)}|D\dot u-D\dot u(x_0)|_{\widetilde u}^2\,dx\right)^{1/2}\mathcal E(\xi)^{1/2}\\
&\le3R\sup_{B_{3R}(x_0)}|D^2\dot u|\left(\int_{B_{3R}(x_0)}\operatorname{tr}\widetilde U\,dx\right)^{1/2}\mathcal E(\xi)^{1/2}\\
&\le CR^{n/2+1}\sup_{B_{3R}(x_0)}|D^2\dot u|\mathcal E(\xi)^{1/2},
\end{aligned}
\]
where we used \eqref{eq:aux00} and \eqref{eq:constrained-local-coefficient-bound} in the last inequality.  Since \(\int_{B_{3R}(x_0)}\rho\,dx\le C\int_{B_{3R}(x_0)}f\,dx\le CR^n\), we have
\[
\begin{aligned}
\left|\int_{B_{3R}(x_0)}\dot\omega f\xi\,dx\right|&\le C\sup_{B_{3R}(x_0)}|\dot\omega|\left(\int_{B_{3R}(x_0)}\rho\,dx\right)^{1/2}\left(\int_{B_{3R}(x_0)}\rho\xi^2\,dx\right)^{1/2}\\
&\le CR^{n/2}\sup_{B_{3R}(x_0)}|\dot\omega|\,\|\xi\|_{L^2(\rho\,dx)}.
\end{aligned}
\]
Then it follows from \eqref{eq:constrained-cutoff-energy} that
\[
|\mathcal E(\dot\zeta,\xi)|\leq CR^{n/2+1}\sup_{B_{3R}(x_0)}\bigl(|D^2\dot u|+|\dot\omega|\bigr)\bigl(Y+R^{-2}J(\dot\zeta;B_{3R}(x_0))\bigr)^{1/2}.
\]
Young's inequality in \eqref{eq:constrained-cutoff-test}, followed by absorption of the term containing $Y$, now gives
\begin{equation}
\label{eq:constrained-cutoff-local-bound}
J(\dot\zeta;B_R(x_0))\leq\frac CR J(\dot\zeta;B_{3R}(x_0))+CR^{n+2}\sup_{B_{3R}(x_0)}\bigl(|D^2\dot u|+|\dot\omega|\bigr)^2.
\end{equation}

\medskip
\noindent\emph{Absorption of the outer-ball energy.} Set
\[
X_R:=\sup_{x_0}J(\dot\zeta;B_R(x_0)),
\]
which is finite by smoothness and periodicity. A covering by balls of radius $R$ gives $X_{3R}\leq C_nX_R$, so \eqref{eq:constrained-cutoff-local-bound} yields
\[
X_R\leq CR^{-1}X_{3R}+Cd^2R^{n+2}\leq CC_nR^{-1}X_R+Cd^2R^{n+2}.
\]
Fix $R_*\geq2R_0$ sufficiently large that $CC_n/R_*\leq1/2$. Absorbing the first term gives $X_{R_*}\leq Cd^2$. For $1\leq R\leq R_*$, inclusion gives $X_R\leq Cd^2$; for $R>R_*$, a finite covering by balls of radius $R_*$ gives $X_R\leq C_Rd^2$. This proves \eqref{eq:constrained-local-response-energy} and completes the proof.
\end{proof}

\begingroup
\begin{Lemma}
\label{lem:constrained-local-multiplier}
Assume \eqref{eq:linearized-background-assumptions}, let $(\dot{\widetilde u},\dot\lambda)$ be a smooth solution of \eqref{eq:linearized-constrained-problem} on $\mathbb R^n$ with smooth data $\dot u,\dot\omega$, and set $\dot\zeta:=\dot{\widetilde u}-\dot u$. There exists a constant $C\geq1$, depending only on $n,\widetilde c,\widetilde C$, such that, for every $R\geq C$,
\begin{equation}
\label{eq:constrained-local-multiplier-bound}
|\dot\lambda(z)|\leq C\mathcal E(\dot\zeta;B_R(\bar x_z))^{1/2}+C\sup_{B_R(\bar x_z)}\bigl(|D^2\dot u|+|\dot\omega|\bigr)\qquad\forall\,z\in\mathbb Z^n.
\end{equation}
\end{Lemma}

\begin{proof}
The pairings in this proof are taken over $\mathbb R^n$. Fix $z$ and choose $B:=B_{C_n}(\bar x_z)$ containing the supports of $\chi_z$ and $\phi_z$. After subtracting the same affine function from $\dot{\widetilde u}$ and $\dot u$, we may assume that $\dot u(\bar x_z)=0$ and $D\dot u(\bar x_z)=0$. Hence $|\dot u|+|D\dot u|\leq C\|D^2\dot u\|_{L^\infty(B)}$ on $B$. By \eqref{eq:constrained-multiplier-recovery}, the kernel bounds in Lemma~\ref{lem:cell-reconstruction}, and \eqref{eq:constrained-local-coefficient-bound}, we obtain
\[
\begin{aligned}
|\dot\lambda(z)|&=\left|\mathcal E(\dot\zeta,\phi_z)+\mathcal E(\dot u,\phi_z)+\langle\dot\omega f,\phi_z\rangle_{L^2}\right|\\
&\leq |\mathcal E(\dot\zeta,\phi_z)|+|\mathcal E(\dot u,\phi_z)|+|\langle\dot\omega f,\phi_z\rangle_{L^2}|\\
&\leq \mathcal E(\phi_z)^{1/2}\bigl(\mathcal E(\dot\zeta;B)^{1/2}+\mathcal E(\dot u;B)^{1/2}\bigr)+\|\dot\omega\|_{L^\infty(B)}\\
&\leq C\mathcal E(\dot\zeta;B)^{1/2}+C\|D^2\dot u\|_{L^\infty(B)}+\|\dot\omega\|_{L^\infty(B)},
\end{aligned}
\]
where we used $\phi_z\geq0$ and $\langle f,\phi_z\rangle_{L^2}=1$. Since $B\subset B_R(\bar x_z)$ for $R\geq C$, this proves \eqref{eq:constrained-local-multiplier-bound}.
\end{proof}
\par\endgroup

We now use the local energy bound to estimate $\dot{\widetilde u}-\dot u$ and $\dot\lambda$ uniformly.

\begin{Lemma}
\label{lem:constrained-linear-response}
Under the assumptions of Lemma~\ref{lem:linearized-constrained-solvability}, there exists a constant $C>0$, depending only on $n,\widetilde c,\widetilde C$, such that the solution $(\dot{\widetilde u},\dot\lambda)$ satisfies
\begin{equation}
\label{eq:constrained-linear-response-bound}
\|\dot{\widetilde u}-\dot u\|_\infty+\|\dot\lambda\|_{\ell^\infty}\leq C\bigl(\|D^2\dot u\|_\infty+\|\dot\omega\|_\infty\bigr).
\end{equation}
\end{Lemma}

\begin{proof}
Denote $\rho=\det D^2\widetilde u$, and set
\[
\dot\zeta=\dot{\widetilde u}-\dot u \quad \text{with}\quad  \Pi[\dot\zeta]=0,\qquad d=\|D^2\dot u\|_\infty+\|\dot\omega\|_\infty.
\]
Combining \eqref{eq:constrained-local-multiplier-bound} with \eqref{eq:constrained-local-response-energy}, we obtain
\[
\|\dot\lambda\|_{\ell^\infty}\leq Cd.
\]
The linearized equation gives $\|\rho^{-1}\mathcal L[\dot{\widetilde u}]\|_\infty\leq Cd$. Working with the periodic lifts, fix $R$ as in Lemma~\ref{lem:linearized-local-pointwise}, and let $\ell_{x_0}$ be the first-order Taylor polynomial of $\dot u$ at $x_0$. Taylor's formula gives $\|\dot u-\ell_{x_0}\|_{L^\infty(B_R(x_0))}\leq Cd$. Together with \eqref{eq:constrained-local-response-energy} and $\int_{B_R(x_0)}\rho\,dx\leq C$, this implies
\[
\|\dot{\widetilde u}-\ell_{x_0}\|_{L^2(B_R(x_0),\rho\,dx)}\leq Cd.
\]
Applying Lemma~\ref{lem:linearized-local-pointwise} and using $\mathcal L[\ell_{x_0}]=0$, we obtain
\[
|(\dot{\widetilde u}-\ell_{x_0})(x_0)|\leq\frac14\|\dot{\widetilde u}-\ell_{x_0}\|_{L^\infty(B_R(x_0))}+Cd,
\]
and hence
\[
|\dot\zeta(x_0)|\leq\frac14\|\dot\zeta\|_\infty+Cd.
\]
Taking the supremum and absorbing the first term gives \eqref{eq:constrained-linear-response-bound}.
\end{proof}

\begin{proof}[Proof of Proposition~\ref{appsharp:blocks}]
Set $d:=\mathcal N(v,\omega)$ and $u_t:=q_A+tv$, and consider
\[
\det D^2\widetilde u_t=\left(\sum_z\lambda_t(z)\chi_z+t\omega\right)f,\qquad \Pi[\widetilde u_t-u_t]=0.
\]
At $t=0$, the convex solution $\widetilde u_0$ of $\det D^2\widetilde u_0=(\det A)f$, with $\widetilde u_0-q_A$ $\mathbb Z^n$-periodic, normalized so that its periodic correction has zero $\mu_0$-mean, satisfies these equations with $\lambda_0\equiv\det A$.

Fix $0<\gamma<1$. At any smooth uniformly convex solution, the linearized operator
\[
\begin{pmatrix}
\mathcal L_{\widetilde u_t}&-\Pi^T\\
\Pi&0
\end{pmatrix}
:C^{2,\gamma}(\mathbb T_L^n)\times\mathbb R^{\Lambda_L}\longrightarrow C^{0,\gamma}(\mathbb T_L^n)\times\mathbb R^{\Lambda_L}
\]
is an isomorphism. Indeed, for prescribed right-hand sides $(g,\dot a)$, we take $\dot\omega=g/f$ and $\dot u=\Phi[\dot a]$ in the variational construction of Lemma~\ref{lem:linearized-constrained-solvability} and then apply Schauder estimates for the fixed coefficients. The implicit function theorem therefore yields a locally unique smooth family of solutions near $t=0$.

Let $[0,t_*)$ be the maximal interval on which this family consists of smooth uniformly convex solutions $\widetilde u_t$ with $\tfrac12\det A<\lambda_t(z)<2\det A$. Assume for contradiction that $t_*\leq1$. For $d\leq\tfrac14\det A$, we have
\[
\frac14(\det A)f\leq\det D^2\widetilde u_t\leq\frac94(\det A)f\qquad\forall\,0\leq t<t_*.
\]
Differentiating the constrained equations in $t$ and applying \eqref{eq:constrained-linear-response-bound} then gives
\[
\|\dot\lambda_t\|_\infty\leq Cd,\qquad\|\lambda_t-\det A\|_\infty\leq Ctd\leq\frac18\det A,
\]
provided $d$ is sufficiently small, depending only on $n$ and $\mathcal K$.
 
For the fixed smooth data, standard periodic Monge--Amp\`ere estimates give a subsequential limit $\widetilde u_{t_*}$ at $t_*$ that is smooth and uniformly convex, with multipliers $\lambda_{t_*}(z)\in[\tfrac78\det A,\tfrac98\det A]$. The implicit function theorem then extends the family past $t_*$, a contradiction. Hence a smooth solution exists at $t=1$.

Uniqueness follows by testing the difference of the determinant equations in \eqref{appweighted:additive-equations} against the difference $\widetilde u-\widehat u$ of two solutions. The average constraints cancel the multiplier terms, and integration by parts shows that $\widetilde u-\widehat u$ is constant. Its zero cell averages give $\widetilde u=\widehat u$, and $\Pi\Phi=I$ then yields equality of the multipliers.

Finally, integrating \eqref{eq:constrained-linear-response-bound} along the segment joining the two data pairs yields \eqref{appweighted:parameter-lipschitz}.
\end{proof}

\subsection{Localized first- and second-variation estimates}
\label{sec5:localized-variations}

We prove the two variation estimates used in Proposition~\ref{appweighted:torus}. The first controls the multiplier of $\dot{\widetilde u}_0-V$ from the Taylor remainders in its data. The second gives $|\ddot\lambda_t(0)|\leq C\varepsilon^2$ when $|D^2\dot u|+|\dot\omega| \leq C\varepsilon(1+|x|)$. We use the notation of Section~\ref{sec5:linearized-estimates}:
\[
\rho:=\det D^2\widetilde u,\qquad(\widetilde U^{ij}):=\operatorname{cof}(D^2\widetilde u),\qquad\mathcal L[w]:=\widetilde U^{ij}w_{ij},\qquad|p|_{\widetilde u}^2:=\widetilde U^{ij}p_ip_j.
\]
For functions on $\mathbb R^n$, write
\[
\begin{aligned}
\mathcal E(w,\zeta)&:=\int_{\mathbb R^n}\widetilde U^{ij}w_i\zeta_j\,dx,\\
\mathcal E(w)&:=\mathcal E(w,w),\\
\mathcal E(w;D)&:=\int_D|Dw|_{\widetilde u}^2\,dx,\\
J(w;D)&:=\mathcal E(w;D)+\int_D\left(\rho+\operatorname{tr}(\widetilde U^{ij})\right)w^2\,dx.
\end{aligned}
\]

We first construct $H_z$ with $\Pi[H_z]=\delta_z$ and exponential decay away from $\bar x_z$. In the following lemma and its proof, spatial integrals and lattice sums are taken over $\mathbb R^n$ and $\mathbb Z^n$, respectively. We write
\[
\ker\Pi:=\{\zeta\in H^1(\mathbb R^n):\Pi[\zeta](z)=0\text{ for every }z\in\mathbb Z^n\}.
\]

\begin{Lemma}
\label{lem:constrained-linear-estimates}
Assume \eqref{eq:linearized-background-assumptions}. For each $z\in\mathbb Z^n$, there exists a unique $H_z\in H^1(\mathbb R^n)\cap C^\infty(\mathbb R^n)$ satisfying
\[
\mathcal E(H_z,\zeta)=0,\qquad \Pi[H_z]=\delta_z\qquad\forall\,\zeta\in\ker\Pi.
\]
There exist constants $c,C>0$, depending only on \(n,\widetilde c,\widetilde C\), such that
\begin{equation}
\label{eq:constrained-h-pointwise}
|H_z(x)|\leq Ce^{-c|x-\bar x_z|}\qquad\forall\,x\in\mathbb R^n.
\end{equation}
\end{Lemma}

\begin{proof}
\begingroup
For compactly supported functions and finitely supported lattice data, we apply \eqref{eq:constrained-reconstruction-bound}, \eqref{appsharp:augmented-gap}, and \eqref{appsharp:global-weighted-trace} after periodizing on $\mathbb T_{NL}^n$ for sufficiently large $N$. Since the constants are independent of $N$, these estimates extend by density to $H^1(\mathbb R^n)$ and $\ell^2(\mathbb Z^n)$.

For the fixed smooth coefficients, coercivity on $\ker\Pi$ gives a unique minimizer
\[
H_z:=\operatorname*{arg\,min}_{\substack{w\in H^1(\mathbb R^n)\\
\Pi[w]=\delta_z}}\frac12\mathcal E(w).
\]
The Euler--Lagrange equation and the reconstruction identities give a multiplier $\lambda_z:\mathbb Z^n\to\mathbb R$ satisfying
\par\endgroup
\begin{equation}
\label{eq:constrained-multiplier-kernel-recovery}
\mathcal L[H_z]=\Pi^T[\lambda_z],\qquad\lambda_z(\widetilde z)=-\mathcal E(H_z,\phi_{\widetilde z}).
\end{equation}
Elliptic regularity yields smoothness. Set $\zeta_z=H_z-\phi_z\in\ker\Pi$. By energy orthogonality and \eqref{eq:constrained-reconstruction-bound},
\[
\mathcal E(H_z)+\mathcal E(\zeta_z)=\mathcal E(\phi_z)\leq C.
\]
Recall the definition of $J$ in \eqref{eq:constrained-local-energy}. Since $\Pi[H_z]=\delta_z$ and $\Pi[\zeta_z]=0$, \eqref{appsharp:augmented-gap} and \eqref{appsharp:global-weighted-trace} give
\[
J(H_z;\mathbb R^n)+J(\zeta_z;\mathbb R^n)\leq C.
\]

Choose a fixed $R_0$ such that the ball $B_{R_0}(\bar x_z)$ contains the supports of $\chi_z$ and $\phi_z$. For $r\geq4R_0$ and $R\geq1$, take a smooth cutoff $\eta$ which vanishes on $B_r(\bar x_z)$, equals one outside $B_{r+R}(\bar x_z)$, and satisfies $|D\eta|\leq C/R$. Set $\xi=\eta^2\zeta_z-\Phi[\Pi[\eta^2\zeta_z]]\in\ker\Pi$. Then $\xi$ vanishes in a neighborhood of $\operatorname{supp}\phi_z$, and the energy orthogonality of $H_z$ gives
\[
\mathcal E(\zeta_z,\xi)=\mathcal E(H_z,\xi)-\mathcal E(\phi_z,\xi)=0.
\]
Write $T(r):=J(H_z;\mathbb R^n\setminus B_r(\bar x_z))$. The cutoff calculation giving \eqref{eq:constrained-cutoff-test} applies on $\mathbb R^n$ and yields
\[
T(r+R)\leq\frac CR T(r-2R_0).
\]
Fixing $R$ sufficiently large and iterating then yields
\begin{equation}
\label{eq:constrained-h-energy-tail}
J(H_z;\mathbb R^n\setminus B_r(\bar x_z))\leq Ce^{-cr}\qquad\forall\,r\geq0.
\end{equation}
Combining \eqref{eq:constrained-multiplier-kernel-recovery} with \eqref{eq:constrained-h-energy-tail} and $\mathcal E(\phi_{\widetilde z})\leq C$, we obtain $|\lambda_z(\widetilde z)|\leq Ce^{-c|\widetilde z-z|}$. Since $f/\rho$ is bounded, the equation $\mathcal L[H_z]=\Pi^T[\lambda_z]$ then gives $|\rho^{-1}\mathcal L[H_z](x)|\leq Ce^{-c|x-\bar x_z|}$.

For the fixed smooth periodic coefficients, local elliptic estimates give $\|H_z\|_\infty<\infty$. Taking the supremum in Lemma~\ref{lem:linearized-local-pointwise} and using \eqref{eq:constrained-h-energy-tail} gives $\|H_z\|_\infty\leq\frac14\|H_z\|_\infty+C$, and hence $\|H_z\|_\infty\leq C$.

Set $m(r):=\sup_{|x-\bar x_z|\geq r}|H_z(x)|$. Applying the pointwise estimate at points with $|x-\bar x_z|\geq r+R$, where $R$ is the fixed radius in that lemma, we obtain
\[
m(r+R)\leq\frac14m(r)+Ce^{-cr}.
\]
By iteration, this proves \eqref{eq:constrained-h-pointwise}.
\end{proof}

The frozen Taylor model has polynomial growth. We therefore extend the local response estimate to data satisfying \eqref{eq:polynomial-response-data}. In the next two lemmas, $R_*\geq1$ is a sufficiently large fixed radius.

\begin{Lemma}
\label{lem:polynomial-linear-response}
Assume \eqref{eq:linearized-background-assumptions}, let $(\dot{\widetilde u},\dot\lambda)$ be a smooth solution of \eqref{eq:linearized-constrained-problem} on $\mathbb R^n$ with smooth data $\dot u,\dot\omega$, and set $\dot\zeta:=\dot{\widetilde u}-\dot u$. Fix $x_*\in\mathbb R^n$ and $\delta>0$. Suppose that
\begin{equation}
\label{eq:polynomial-response-data}
|D^2\dot u(x)|+|\dot\omega(x)|\leq\delta(1+|x-x_*|)^2\qquad\forall\,x\in\mathbb R^n,
\end{equation}
and that $\dot\zeta$ and $D\dot\zeta$ grow at most quadratically. Let $\ell_{x_0}$ be the first-order Taylor polynomial of $\dot u$ at $x_0$. Then there exists a constant $C>0$, depending only on $n,\widetilde c,\widetilde C$ and $R_*$, such that
\begin{equation}
\label{eq:polynomial-linear-response}
J(\dot\zeta;B_{R_*}(x_0))+\mathcal E(\dot{\widetilde u}-\ell_{x_0};B_{R_*}(x_0))\leq C\delta^2(1+|x_0-x_*|)^4\qquad\forall\,x_0\in\mathbb R^n,
\end{equation}
and
\begin{equation}
\label{eq:polynomial-multiplier-response}
|\dot\lambda(z)|\leq C\delta(1+|\bar x_z-x_*|)^2\qquad\forall\,z\in\mathbb Z^n.
\end{equation}
\end{Lemma}

\begin{proof}
We repeat the proof of \eqref{eq:constrained-local-response-energy}, using compactly supported tests on $\mathbb R^n$ and taking the supremum after division by $(1+|x_0-x_*|/R_*)^4$. By \eqref{eq:polynomial-response-data}, the last term in \eqref{eq:constrained-cutoff-local-bound} is bounded by $C_{R_*}\delta^2(1+|x_0-x_*|/R_*)^4$. Since $\dot\zeta$ and $D\dot\zeta$ grow at most quadratically,
\[
\sup_{x_0}\frac{J(\dot\zeta;B_{R_*}(x_0))}{(1+|x_0-x_*|/R_*)^4}<\infty.
\]
The factors $1+|x_0-x_*|/R_*$ at centers within distance $3R_*$ are uniformly comparable. The covering argument in Lemma~\ref{lem:constrained-local-response-energy}, with $R_*$ sufficiently large, therefore gives
\[
\begin{aligned}
\sup_{x_0}\frac{J(\dot\zeta;B_{R_*}(x_0))}{(1+|x_0-x_*|/R_*)^4}&\leq\frac12\sup_{x_0}\frac{J(\dot\zeta;B_{R_*}(x_0))}{(1+|x_0-x_*|/R_*)^4}+C\delta^2.
\end{aligned}
\]
Absorbing the first term gives the bound for $J$ in \eqref{eq:polynomial-linear-response}.

On $B_{R_*}(x_0)$, Taylor's formula gives $|D\dot u-D\ell_{x_0}| \leq C_{R_*}\delta(1+|x_0-x_*|)^2$. Since $\dot{\widetilde u}-\ell_{x_0}=\dot\zeta+(\dot u-\ell_{x_0})$, the bound for $J$ and \eqref{eq:constrained-local-coefficient-bound} then give the estimate for $\mathcal E(\dot{\widetilde u}-\ell_{x_0};B_{R_*}(x_0))$ in \eqref{eq:polynomial-linear-response}. Finally, \eqref{eq:constrained-local-multiplier-bound}, together with \eqref{eq:polynomial-response-data} and \eqref{eq:polynomial-linear-response}, gives \eqref{eq:polynomial-multiplier-response}.
\end{proof}

We next estimate $\ddot\lambda_t$ along affine paths $(u_t,\omega_t,A_t)$. Testing the twice-differentiated equation against $H_z$ eliminates $\ddot{\widetilde u}_t$; the remaining quadratic term is controlled by the preceding energy estimate.

\begin{Lemma}
\label{lem:localized-second-variation}
Let $(\widetilde u_t,\lambda_t)$ be a smooth family satisfying \eqref{appweighted:additive-equations} with background $u_t$ and additive density $\omega_t$, where $u_t,\omega_t,A_t$ depend affinely on $t$ and $\widetilde u_t-q_{A_t},u_t-q_{A_t}\in C^\infty(\mathbb T_L^n)$ for a fixed $L$. Fix $x_*\in\mathbb R^n$ and $\delta>0$. Assume that \eqref{eq:linearized-background-assumptions} holds at time $t$ and that the variations $\dot u_t,\dot\omega_t$ satisfy \eqref{eq:polynomial-response-data}. Then there exists a constant $C>0$ such that
\begin{equation*}
|\ddot\lambda_t(z)|\leq C\delta^2(1+|\bar x_z-x_*|)^4\qquad\forall\,z\in\mathbb Z^n.
\end{equation*}
\end{Lemma}

\begin{proof}
We fix $t$. Throughout the proof, $\rho$, $\widetilde U$, $\mathcal L$, and $H_z$ are evaluated at this fixed $t$.

\medskip
\noindent\emph{Step 1. Elimination of the second variation of the potential.}
Differentiating \eqref{appweighted:additive-equations} gives \eqref{eq:linearized-constrained-problem}. Since $\dot\zeta_t=\dot{\widetilde u}_t-\dot u_t\in\ker\Pi$ is smooth and periodic, Lemma~\ref{lem:polynomial-linear-response} applies. For a smooth function $V$, consider the second variation
\[
\mathcal R[V]:=\left.\frac{d^2}{ds^2}\det D^2(\widetilde u_t+sV)\right|_{s=0}.
\]
Using $\mathcal L[V]=\widetilde U^{ij}V_{ij}$, we have
\begin{equation}
\label{eq:second-variation-positive-part}
\mathcal R[V]=\frac{1}{\rho}\left[(\widetilde U^{ij}V_{ij})^2-\widetilde U^{ik}\widetilde U^{jl}V_{ij}V_{kl}\right]\leq\frac{|\mathcal L[V]|^2}{\rho}.
\end{equation}
Differentiating \eqref{appweighted:additive-equations} twice with respect to $t$ yields
\[
\mathcal L[\ddot{\widetilde u}_t]+\mathcal R[\dot{\widetilde u}_t]=\Pi^T[\ddot\lambda_t],\qquad \Pi[\ddot{\widetilde u}_t]=0.
\]
The function $\ddot{\widetilde u}_t$ is periodic since $\ddot A_t=0$. By \eqref{eq:constrained-multiplier-kernel-recovery} and the decay estimates \eqref{eq:constrained-h-pointwise} and \eqref{eq:constrained-h-energy-tail}, integration by parts with a cutoff gives
\[
\int_{\mathbb R^n}H_z\mathcal L[\ddot{\widetilde u}_t]\,dx=\int_{\mathbb R^n}\ddot{\widetilde u}_t\mathcal L[H_z]\,dx=\sum_{\widetilde z\in\mathbb Z^n}\lambda_z(\widetilde z)\Pi[\ddot{\widetilde u}_t](\widetilde z)=0.
\]
Multiplying the equation for $\ddot{\widetilde u}_t$ by $H_z$ and using $\Pi[H_z]=\delta_z$ and \eqref{eq:constrained-h-pointwise}, we then obtain
\[
\ddot\lambda_t(z)=\int_{\mathbb R^n}H_z(x)\mathcal R[\dot{\widetilde u}_t](x)\,dx.
\]
Thus, by \eqref{eq:constrained-h-pointwise}, it suffices to prove
\begin{equation}
\label{eq:localized-second-variation-mass}
\int_{B_{R_*}(x_0)}|\mathcal R[\dot{\widetilde u}_t]|\,dx\leq C\delta^2(1+|x_0-x_*|)^4\qquad\forall\,x_0\in\mathbb R^n.
\end{equation}
The case $n=1$ is immediate, since $\mathcal R[V]=0$. We therefore assume $n\geq2$.

\medskip
\noindent\emph{Step 2. A bound for the absolute value of the quadratic term.}
The upper bound in \eqref{eq:second-variation-positive-part} controls $\mathcal R[V]_+$. To control the negative part after integration, define
\[
\mathsf M[V]^{ij}:=\frac{1}{\rho}\bigl(\widetilde U^{ij}\widetilde U^{kl}-\widetilde U^{ik}\widetilde U^{jl}\bigr)V_kV_l.
\]
This matrix is positive semidefinite and satisfies
\begin{equation}
\label{eq:second-variation-double-divergence}
\mathcal R[V]=-\partial_i\partial_j\mathsf M[V]^{ij},\qquad\mathsf M[V]:D^2\widetilde u_t=(n-1)|DV|_{\widetilde u_t}^2.
\end{equation}
The first identity follows by differentiating $\mathsf M[V]$ twice: the derivatives of the mixed cofactors vanish after contraction, and their antisymmetry cancels the third derivatives of $V$.

Fix $x_0$, let $\ell_{x_0}$ be the first-order Taylor polynomial of $\dot u_t$ at $x_0$, and set $V=\dot{\widetilde u}_t-\ell_{x_0}$. By Lemma~\ref{lem:cell-average-convex-bound}, we may choose a constant $C_0$ such that $h_t=\widetilde u_t-u_t+C_0$ satisfies $0\leq h_t\leq C$ and $|Dh_t|\leq C$. Set
\[
\widetilde\eta(x):=\left(1+\frac{|x-x_0|^2}{R_*^2}\right)^{-(n+4)},\qquad \Psi:=\widetilde\eta(1+h_t).
\]
The bounds $|D\widetilde\eta|\leq CR_*^{-1}\widetilde\eta$, $|D^2\widetilde\eta|\leq CR_*^{-2}\widetilde\eta$ and $D^2u_t\geq\widetilde cI$ give, for sufficiently large $R_*$,
\[
D^2\Psi =\widetilde\eta(D^2\widetilde u_t-D^2u_t)+(1+h_t)D^2\widetilde\eta+D\widetilde\eta\otimes Dh_t+Dh_t\otimes D\widetilde\eta \leq\widetilde\eta D^2\widetilde u_t.
\]
Integrating \eqref{eq:second-variation-double-divergence} by parts twice on $\mathbb R^n$, we obtain
\[
\begin{aligned}
\int_{\R^n}\widetilde\eta|\mathcal R[V]|\,dx&\leq\int_{\R^n}\Psi|\mathcal R[V]|\,dx\\
&=\int_{\R^n}\Psi\left(2\mathcal R[V]_+-\mathcal R[V]\right)\,dx\\
&=2\int_{\R^n}\Psi\mathcal R[V]_+\,dx+\int_{\R^n}\mathsf M[V]:D^2\Psi\,dx\\
&\leq C\int_{\R^n}\widetilde\eta\mathcal R[V]_+\,dx+\int_{\R^n}\widetilde\eta\mathsf M[V]:D^2\widetilde u_t\,dx\\
&= C\int_{\R^n}\widetilde\eta\mathcal R[V]_+\,dx+(n-1)\int_{\R^n}\widetilde\eta|DV|_{\widetilde u_t}^2\,dx.
\end{aligned}
\]
The integrations by parts are justified by a standard cutoff argument, since the coefficients are smooth and periodic and $DV$ grows at most linearly.

By \eqref{eq:linearized-constrained-problem}, \eqref{eq:polynomial-multiplier-response}, and \eqref{eq:second-variation-positive-part},
\[
\mathcal R[V]_+\leq C\delta^2(1+|x-x_*|)^4\rho.
\]
Note that
\[
\widetilde\eta(x)\leq C2^{-2(n+4)j},\qquad(1+|x-x_*|)^4\leq C2^{4j}(1+|x_0-x_*|)^4
\]
on the dyadic annuli $B_{2^{j+1}R_*}(x_0)\setminus B_{2^jR_*}(x_0)$, $j\geq0$. By \eqref{eq:constrained-local-coefficient-bound}, summing over $B_{R_*}(x_0)$ and these annuli gives
\[
\begin{aligned}
\int_{\R^n}\widetilde\eta\mathcal R[V]_+\,dx&\leq C\delta^2\int_{\R^n}\widetilde\eta(x)(1+|x-x_*|)^4\rho(x)\,dx\\
&\leq C\delta^2(1+|x_0-x_*|)^4\left(1+\sum_{j\geq0}2^{-(n+4)j}\right)\\
&\leq C\delta^2(1+|x_0-x_*|)^4.
\end{aligned}
\]
Since $DV=D\dot\zeta_t+D\dot u_t-D\dot u_t(x_0)$, we have
\[
|DV|_{\widetilde u_t}^2\leq 2|D\dot\zeta_t|_{\widetilde u_t}^2+2|D\dot u_t-D\dot u_t(x_0)|^2\operatorname{tr}(\widetilde U^{ij}).
\]
By \eqref{eq:polynomial-response-data},
\[
|D\dot u_t(x)-D\dot u_t(x_0)|\leq C\delta(1+|x_0-x_*|)^2(1+|x-x_0|)^3.
\]
For $j\geq0$, cover $B_{2^{j+1}R_*}(x_0)$ by at most $C2^{nj}$ balls of radius $R_*$. Applying \eqref{eq:polynomial-linear-response} to the first term and \eqref{eq:constrained-local-coefficient-bound} to the second gives
\[
\begin{aligned}
\int_{B_{2^{j+1}R_*}(x_0)}|DV|_{\widetilde u_t}^2\,dx&\leq C\delta^2(1+|x_0-x_*|)^4\left(2^{(n+4)j}+2^{(n+6)j}\right)\\
&\leq C\delta^2(1+|x_0-x_*|)^4\,2^{(n+6)j}.
\end{aligned}
\]
Summing over $B_{R_*}(x_0)$ and the same annuli, we obtain
\[
\begin{aligned}
\int_{\R^n}\widetilde\eta|DV|_{\widetilde u_t}^2\,dx\leq C\delta^2(1+|x_0-x_*|)^4\left(1+\sum_{j\geq0}2^{(n+6-2(n+4))j}\right)\leq C\delta^2(1+|x_0-x_*|)^4.
\end{aligned}
\]
The preceding calculation therefore gives the same bound for $\int\widetilde\eta|\mathcal R[\dot{\widetilde u}_t]|\,dx$. Since $\widetilde\eta\geq c$ on $B_{R_*}(x_0)$, this proves \eqref{eq:localized-second-variation-mass} and completes the proof.
\end{proof}

\bigskip

\noindent \textbf{Acknowledgement.} Part of this work was carried out while the second and third named authors were visiting the Hong Kong University of Science and Technology (HKUST) for the IAS Focus Program ``Recent Advances on Elliptic and Parabolic Equations'' in May 2025. They would like to thank HKUST for its warm hospitality.

\bigskip

\noindent \textbf{AI usage statement.} We originally proved Theorem \ref{thm:main-result} in 2025. After completing our work \cite{jin2025liouville} in the general measure setting, we first sought to extend the case $m=0$ of Theorem \ref{thm:main-result} to the measure setting as well. We recognized that our approach required the Lipschitz estimate established in Theorem \ref{thm:measure-corrector-lipschitz}, but at the time we did not know how to prove it. Following the release of GPT-5.6 Sol, we used this model, through extensive iterative interactions, to develop a proof of Theorem \ref{thm:measure-corrector-lipschitz}. The Lipschitz estimate, together with our third-order corrector construction, then motivated us to investigate the case $m=1$ of Theorem~\ref{thm:weighted-measure-holder-rate}. A proof of this case was obtained on 13 September 2026 through sustained, iterative discussions with GPT-6 Astra. But it was very long and difficult to follow, and it involved probabilistic arguments. Through continued discussions with GPT-6 Astra, we have since substantially reworked the proof, developing alternative arguments and simplifying its overall structure. We wrote the manuscript after working out the proofs and used GPT-6 Astra to help polish the language.

\bigskip

\noindent T. Jin

\noindent Department of Mathematics, The Hong Kong University of Science and Technology\\
Clear Water Bay, Kowloon, Hong Kong\\
Email: \textsf{tianlingjin@ust.hk}

\bigskip

\noindent Y.Y. Li

\noindent Department of Mathematics, Rutgers University\\
110 Frelinghuysen Road, Piscataway, NJ 08854-8019, USA\\
Email: \textsf{yyli@math.rutgers.edu}

\bigskip

\noindent H.V. Tran

\nopagebreak[4]
\noindent Department of Mathematics, University of Wisconsin at Madison\\
Van Vleck Hall, 480 Lincoln Drive, Madison, WI 53706, USA\\
Email: \textsf{hung@math.wisc.edu}

\bigskip

\noindent X. Tu

\noindent Department of Mathematics, The Hong Kong University of Science and Technology\\
Clear Water Bay, Kowloon, Hong Kong\\[1mm]
Emails: \textsf{maxstu@ust.hk}; \textsf{tuxushan28561m@gmail.com}

\end{document}